%% file: CharPoints_Revised_May_2011.tex
\documentclass{amsart}
\usepackage{amsfonts, amsmath, amsthm}
\usepackage{amssymb}

\newcommand{\bfy}{{\mathbf{y}}}

\include{symb1}

\title{Characteristic Points of Recursive Systems}
\author{Jason P. Bell, Stanley N. Burris, and Karen A. Yeats}
\begin{document}

\maketitle



\begin{abstract} 								
Characteristic points have been a primary tool in the study of a 		generating function defined by a single recursive equation. 			
We investigate the proper way to adapt this tool when working with 		
multi-equation recursive systems. 
						
Given an irreducible non-negative power series system with $m$ equations, 
let $\rho$ be the radius of convergence of the solution power series and 
let $\pmb{\tau}$ be the values of the solution series evaluated at $\rho$.
The main results of the paper include: 			
\begin{thlist} 									
\item 									
the set of characteristic points form an antichain in $\bbR^{m+1}$, 	
\item 									
given a characteristic point $(a,\mathbf{b})$, 					
(i) the spectral radius of the Jacobian of $\bG$ at $(a, \mathbf{b})$  is $\ge 1$, 			
and (ii) it is $=1$ iff $(a,\mathbf{b}) = (\rho,\pmb{\tau})$, 			
\item 									
if $(\rho,\pmb{\tau})$ is a characteristic point, then (i) $\rho$ is the largest $a$ 		
for $(a,\mathbf{b})$ a characteristic point, and (ii) a characteristic point	
$(a,\mathbf{b})$ with $a=\rho$ is the extreme point $(\rho,\pmb{\tau})$.			
\end{thlist} 									 				
\end{abstract}


\section{Introduction and Preliminaries}

Recursively defined generating functions play a major role in combinatorial enumeration; see 
the recently published book \cite{FlSe2009} for numerous examples. 
The important technique of expressing a generating function as a product of geometric series (as well as other kinds of products) was introduced by Euler in the mid 1700s, in his study of various problems connected with the number of partitions of integers. This investigation of partition problems was continued by Sylvester and Cayley (see, for example, \cite{Cayley1856}, \cite{Sylvester1857}), starting in the mid 1850s.  The expressions they used for partition generating functions were explicit, whereas the fundamental equation
\begin{eqnarray} \label{cayley form}
\sum_{n\ge 1} t_n x^n & =& x\cdot \prod_{n\ge 1} (1-x^n)^{-t_n},
\end{eqnarray}
 introduced in 1857 by Cayley \cite{Cayley1857}, for rooted unlabeled trees, 
defined the coefficients $t_n$ implicitly, yielding a recursive procedure to compute
the $t_n$.
Cayley used this to recursively calculate (with some errors) the first dozen values of $t_n$, and later applied his method to
recursively enumerate certain kinds of chemical compounds.

	Let $T(x) = \sum_{n\ge 1} t_nx^n$. In 1937 P\'olya (see \cite{PoRe1987})
converted \eqref{cayley form} into 
\begin{eqnarray}\label{polya form}
T(x) & =& x\cdot\exp\Big(\sum_{m\ge 1} T(x^m)/m\Big),
\end{eqnarray}
a form to which he was able to apply analytic techniques to find asymptotics for the $t_n$, namely he proved
\begin{eqnarray}
t_n &\sim& C \rho^{-n}n^{-3/2}
\end{eqnarray}
where $\rho$ is the radius of convergence of $T(x)$, and $C$ a positive constant.\footnote{In \cite{BBY2006} 
 we found this law so ubiquitous among naturally defined classes of trees defined by a single 
 equation that we referred to it as the \textit{universal law} for rooted trees.} 
A similar result held for the various classes
of chemical compounds studied by Cayley. Although the function $T(x)$ was not expressible in terms of well-known functions,
nonetheless P\'olya showed how to determine $C$ and $\rho$ directly from \eqref{polya form}.
 P\'olya's methods were applied to nearly regular classes of trees in 1948 by Otter \cite{Otter1948}. 
 
In 1974 Bender \cite{Bender1974}, following P\'olya's ideas,  formulated a general result for how to determine the 
radius of convergence $\rho$ of a power series $T(x)$ defined by a functional equation $F(x,y)=0$. 
Bender's hypotheses guaranteed that $\rho$ was positive and finite, and that $\tau := T(\rho)$ was also finite.
His method was simply to find $(\rho,\tau)$ among the solutions $(a,b)$ (called \textit{characteristic points}) of the 
\textit{characteristic system}
\begin{eqnarray*}
F(x,y) & = & 0 \\
\frac{\partial F}{\partial y}(x,y) & = & 0.
\end{eqnarray*}
A decade later Canfield \cite{Canfield1984} found a gap in the hypotheses of Bender's formulation when there were several characteristic points. In the case of a polynomial functional equation, Canfield sketched a method 
to determine which of the characteristic points gives the radius of convergence of the solution $y=T(x)$. 

In the late 1980s Meir and Moon \cite{MeMo1989} focused on a special case of Canfield's work, namely
 when $F(x,y)=0$ is of the form $y=G(x,y)$, 
where $G(x,y)$ is a power series with nonnegative coefficients.  
The interesting cases were such that setting 
$T(x) = G(x,T(x))$, with $T(x)$ an indeterminate power series, gave a \textit{recursive} determination of the coefficients 
of $T(x)$. 
One advantage of their restricted form of recursive equation was that there could be at most one characteristic point.
This formulation was adopted by Odlyzko in his 1995 survey paper \cite{Odlyzko1995} as well as in the 
recent book \cite{FlSe2009} of Flajolet and Sedgewick.  These publications have focused on characteristic points in the 
\textit{interior} of  the domain of convergence  of $G(x,y)$, in the context of proving that 
$\rho$ is a square  root singularity of the solution $y = T(x)$. If $(\rho,\tau)$ is on the boundary of the domain of $G(x,y)$  then $\rho$ may not be a square-root singularity of $T(x)$. 

Most areas of application actually require a recursive system
of equations
 \begin{equation}\label{system}
\left\{
\begin{array}{l c l}
y_1&=&G_1(x,y_1,\ldots,y_m)\\
&\vdots&\\
y_m&=&G_m(x,y_1,\ldots,y_m),
\end{array}
\right.
\end{equation}
written more briefly as $\mathbf{y} = \bG(x,\mathbf{y})$. 
(A precise definition of the systems considered in this paper is given in $\S$\ref{sec WellConditioned}.)
This rich area of enumeration has been rather 
slow in it development. In the 1970s Berstel and Soittola (see \cite{FlSe2009} V.3) carried out a thorough analysis of enumerating the 
words in a regular language using recursive systems of equations that were linear in $y_1,\ldots, y_m$. 
However it was not until the 1990s that publications started appearing that used multi-equation non-linear
systems. Following the trend with single recursion equations $y=G(x,y)$, the focus has been on systems 
$\mathbf{y} = \bG(x,\mathbf{y})$ where the $G_i(x,\mathbf{y})$ are power series with non-negative
coefficients.

In 1993 Lalley \cite{Lalley1993} considered polynomial systems in his study of random walks on free groups. In 1997 Woods 
\cite{Woods1997} used one particular system to analyze the asymptotic densities of monadic second-order definable
 classes of trees in the class 
of all trees. In the same year Drmota \cite{Drmota1997} extended Lalley's results to power series systems. 
Lalley's and Drmota's results were for a wide range of irreducible systems, that is, systems in which each variable
$y_i$ (eventually) depends on any variable $y_j$. An irreducible system of the kind they studied behaves in some ways 
like a single equation system, for example, the standard solution $y_i = T_i(x)$ is such that all the $T_i(x)$ have the same 
finite positive radius $\rho$,  the $\tau_i := T_i(\rho)$ are all finite, and the asymptotics for the coefficients of $T_i(x)$
is of the P\'olya form $C_i \rho^{-n}n^{-3/2}$. 

    Thus, as has been the case with single equation systems, it is desirable to find the radius of convergence $\rho$
  even though the solutions $T_i(x)$ may be fairly intractable.  The natural method was to extend the definition of the
  characteristic system from a single equation to a system of equations,  by adding the determinant of the Jacobian  
  of the system, set equal to zero to, to the original system. The solutions of such a characteristic system will again 
  be called characteristic points. 
  
  Under suitable conditions one can find $(\rho,\pmb{\tau})$ among the characteristic points. 
  To-date, however, the necessary study of characteristic points $(a,\mathbf{b})$ for systems, so 
  that one can locate $(\rho,\pmb{\tau})$, has been essentially non-existent. Filling this void is the 
  goal of this paper. In December, 2007, we discovered, in the polynomial systems studied by Flajolet and Sedgewick, and thus in the
more general systems studied by Drmota,
 that it was possible for there to be more than one characteristic point --- 
  this was communicated to Flajolet and appears as an example in \cite{FlSe2009} (p.~484). 
  The main objective of this paper is to give conditions to locate $(\rho,\pmb{\tau})$ among the 
  characteristic points, if indeed $(\rho,\pmb{\tau})$ is a characteristic point. A review of, and 
  improvements to, the theory of the single equation case 
  (see Proposition \ref{1 eq CP} and Corollary \ref{simple sys}) are also
  given.

  It turns out that, even if there is a characteristic point of a system $\mathbf{y} = \bG(x,\mathbf{y})$ 
   in the interior of the domain of $\bG(x,\mathbf{y})$, one cannot claim that the asymptotics for the 
   coefficients of the solutions $T_i(x)$ will be of the above P\'olya form 
  (see Examples \ref{counter Drmota C1=0}, \ref{counter Drmota C1>0}).\footnote{In 1997 
  Drmota \cite{Drmota1997} appears to claim that having a characteristic point in the interior of 
  the domain would lead to P\'olya asymptotics---however these examples show this not to be the case. 
  In his 2009 book \cite{Drmota2009} this hypothesis is replaced with one regarding minimal characteristic 
  points, which seems somewhat at odds with our Proposition \ref{antichain}, which says that the characteristic points 
  form an antichain with the characteristic point $(a,\mathbf{b})$ of interest having the \textit{largest} 
  value of $a$ among the characteristic points. Theorem \ref{drmota} of $\S$\ref{Drmota revisited} 
  is a restatement of Drmota's result, to make it clear which characteristic point is of interest, namely
  the one (if it exists) such that the Jacobian of $\bG(x,\mathbf{y})$ has 1 as its largest real eigenvalue. 
  \label{Drmota footnote}}
 
  We do not investigate the case when $(\rho, \pmb{\tau})$ is not a characteristic point, concluding only that it must be on the boundary of the domain of $\mathbf{G}(x,\mathbf{y})$ and that the spectral radius of the Jacobian of $\bG(x,\bfy)$ at $(\rho, \pmb{\tau})$ is $<1$.  Note that for polynomial systems, $(\rho, \pmb{\tau})$ is always a characteristic point, and in general the spectral
  radius condition (see Lemma \ref{ev of J(x,Tx)}) makes it possible to recognize when $(\rho, \pmb{\tau})$ is among the characteristic points.

\subsection{Outline}
Appendix \ref{background} discusses standard background and notation for power series, including a statement, Proposition \ref{PF Prop}, of the key results of Perron-Frobenius theory.

Section \ref{sec WellConditioned} sets up the equational systems of interest.  Section \ref{Char section} begins by reducing to the case where the Jacobian matrix $J_\bG(x,\bfy)$ has nonzero entries and then proceeds to the more interesting discussion of properties of characteristic points, including notably Proposition \ref{antichain}. This leads to the main result of the section, Theorem \ref{location}, followed by the single equation result, Proposition \ref{1 eq CP}.  Section \ref{sec eigenpoints} introduces an eigenvalue criterion for critical points leading to the main result of the paper, Theorem \ref{RCP Thm}.  Section \ref{Drmota sec} then uses the preceding results to correct an inaccuracy in the literature.  The main body of the paper concludes with some open problems.

Appendix \ref{ex section} contains a large number of examples illustrating the various possibilities and results.  It is best read along side the main body of the paper.

\section{Well-conditioned systems}\label{sec WellConditioned}
The next definition gives a version of essentially well-known conditions which ensure that a system
 $\mathbf{y}=\bG(x,\mathbf{y})$ as in \eqref{system}
 has  power series solutions $y_i=T_i(x)$ of the type encountered in generating functions 
 for classes of trees. (See Drmota \cite{Drmota1997}, \cite{Drmota2009}.)
 
\begin{definition}\label{WellConditioned}
A system $\mathbf{y}=\bG(x,\mathbf{y})$ is \textit{well-conditioned} if it satisfies
\begin{thlist}
\item
each $G_i(x,\mathbf{y})$ is a power series with nonnegative coefficients
\item 
$\bG(x,\mathbf{y})$ is holomorphic in a neighborhood of the origin
\item
$\bG(0,\mathbf{y}) = \mathbf{0}$
\item for all $i$,
$G_i(x,\mathbf{0}) \neq  0$ 
\item
the system is irreducible\footnote{This means the non-negative matrix $J_\bG$ is irreducible.}
\item
for some $i,j,k$,
$\displaystyle 
\frac{\partial^{2}G_i(x,\mathbf{y})}{\partial {y_j}\partial y_k}\ \neq\ 0$
(so the system is nonlinear in $\mathbf{y}$).
\end{thlist}
\end{definition}

\begin{remark}
Since $\bG(x,\bfy)$ has non-negative coefficients, condition {\rm(b)} is equivalent to  {\rm(b$'$)}: $\bG(x,\mathbf{y})$ converges at some positive 
$(a,\mathbf{b})$.
\end{remark}

\subsection{Solutions of Well-Conditioned Systems}

The following proposition is standard.
\begin{proposition}\label{WellKnown}
If $\mathbf{y}=\bG(x,\mathbf{y})$ is a well-conditioned system then
the following hold: 
\begin{itemize}
\item[(i)]
There is a unique vector $\bT(x)$ of formal power series $T_i(x)$ with nonnegative coefficients such that one has the formal identity
\begin{equation}\label{the sol}
\bT(x)\ =\ \bG\big(x,\bT(x)\big).
\end{equation}
\item[(ii)]
Equation \eqref{the sol} gives a recursive procedure to find the coefficients of the $T_i(x)$.
\item[(iii)]
Equation \eqref{the sol} holds for $x\in[0,\infty]$.
\item[(iv)]
All $T_i(x)$ have the same radius of convergence $\rho\in(0,\infty)$ and all
$T_i(x)$ converge at $\rho$, that is,
$\tau_i :=T_i(\rho)<\infty$.
  \item[(v)]
Each $T_i(x)$ has a singularity at $x=\rho$.
\item[(vi)]
If $(\rho,\pmb{\tau})$ is in the
 interior of the domain of $\bG(x,\mathbf{y})$ then 
$$
   \det\big(I - J_\bG(\rho,\pmb{\tau})\big)\ =\ 0.
$$
\end{itemize}
\end{proposition}

\begin{proof}Apply Proposition \ref{permanence}, Pringsheim's Theorem, and the Implicit Function Theorem.
\end{proof}


The sequence $\bT(x)$ of power series described 
in Proposition \ref{WellKnown} is
the \textit{standard solution} of the system, and the point $(\rho,\pmb{\tau})$ is the
\textit{extreme point} (of the standard solution, or of the system). 
{}From \eqref{the sol} one has
$\bT(0) = \mathbf{0}$, so the standard solution goes through the origin.
The set 
$$
\domp(\bG)\ :=\ \big\{(a,\mathbf{b}) : a,b_1,\ldots,b_m > 0\text{ and } G_i(a,\mathbf{b}) < \infty, 1\le i\le m
\big\}
$$
is the \textit{positive domain} of $\bG$. For $(a,\mathbf{b})\in\domp(\bG)$ let
$$
\Lambda(a,\mathbf{b})\ :=\ \Lambda\big(J_\bG(a,\mathbf{b})\big),
$$
the largest real eigenvalue of the Jacobian matrix $J_\bG(a,\mathbf{b})$.
Since $J_\bG(a,\mathbf{b})$ is a matrix with non-negative entries, 
$\Lambda(a,\mathbf{b})$ is the spectral radius of $J_\bG(a,\mathbf{b})$.
 
 
 \subsection{Characteristic Systems, Characteristic Points}
 
Flajolet and Sedgewick \cite{FlSe2009} VII.6 define the \textit{characteristic system}  
of \eqref{system} to be 
\[
\left\{
\begin{array}{l c l}
y_1&=&G_1(x,y_1,\ldots,y_m)\\
&\vdots&\\
y_m&=&G_m(x,y_1,\ldots,y_m)\\
0&=&  \det\big(I - J_\bG(x,\mathbf{y})\big).
\end{array}
\right.
\]
Let the positive solutions $(a,\mathbf{b})\in \mathbb{R}^{m+1}$ to this system be called the
 \textit{characteristic points} of the system.\footnote{Flajolet and Sedgewick 
 (\cite{FlSe2009} Chapter VII p.~468) only consider
 characteristic points in the interior of $\domp(\bG)$.} 
 Requiring that $(\rho,\pmb{\tau})$ be a characteristic point in the interior of the 
 domain of $\bG(x,\mathbf{y})$ has been crucial to proofs that $x=\rho$ 
 is a square-root singularity of the 
 $T_i(x)$, leading to the asymptotics 
$t_i(n)\ \sim\ C_i\rho^{-n}n^{-3/2}$ for the non-zero coefficients. 
There is, thus, considerable interest in finding practical computational means of estimating 
$\rho$.

For the case that the $G_i(x,\mathbf{y})$ are \textit{polynomials} we know that
$(\rho,\pmb{\tau})$ will be among the characteristic points and in the interior of
 the domain of $\bG$. 
However until now, even in the polynomial case, no general 
attempt has been made to characterize $(\rho,\pmb{\tau})$ among the characteristic points of the system\footnote{When dealing with polynomial systems in Chapter VII of \cite{FlSe2009}, Flajolet and Sedgewick do not use
characteristic systems---they prefer to work with the singularities, and their connections via
branches, of the algebraic curves $y_i(x)$ defined by the system.
}---with one exception, namely  
the 1-equation systems. 


\section{Characteristic Points of Well-Conditioned Systems}\label{Char section}
 
{}From now on it is assumed, unless stated otherwise, that we are working with a 
well-conditioned system $\Sigma: \mathbf{y} = \bG(x,\mathbf{y})$ of $m$ equations.

\subsection{Making substitutions in an irreducible system}

 A careful analysis of the characteristic points of $\Sigma$
  is easier if $J_\bG(a,\mathbf{b})$ is a positive matrix for positive
 points $(a,\mathbf{b})$; this is the case precisely when no entry of $J_\bG(x,\mathbf{y})$ is 0.
Fortunately there is a substitution procedure to transform the original system 
$\Sigma$ into 
a well-conditioned system $\Sigma^\star$ with
\begin{thlist}
\item[i]
exactly the same positive solutions  $(a,\mathbf{b})$, and 
\item[ii] 
exactly the same set $\mathcal{CP}$ of characteristic points,
\end{thlist}
and such that for the new system  $\bfy = \bG^\star(x,\bfy)$,
the Jacobian $J_{\bG^\star}(x,\mathbf{y})$  has no zero entries.
 Indeed, given any positive integer $n$, one can
 carry out the substitutions so that all $n$th partial derivatives of $\bG(x,y)$
 with respect to the $y_i$ are non-zero. The goal of this section is to prove
 these claims.
 
The simplest substitutions are $n$-fold iterations 
$\bG^{(n)}$ of the transformation $\bG$.  These are used in \cite{FlSe2009} (see p.~492) as they suffice for
\textit{aperiodic}\footnote{A well-conditioned system
$\mathbf{y} = \bG(x,\mathbf{y})$ is \textit{aperiodic} if the coefficients of 
each $T_i(x)$ are eventually positive, $\bT(x)$ being the standard solution---see
\cite{FlSe2009}, p.~489.}
polynomial systems $\Sigma$.  In general, however, iteration of $\bG$
 does not suffice to obtain a
system $\Sigma^\star$ as described above---see Example \ref{ex subst}.

Given a system $\Sigma: \mathbf{y}=\bG(x,\mathbf{y})$, a \textit{minimal self-substitution transformation} creates the
system $\Sigma^{(\alpha)}: \mathbf{y}=\bG^{(\alpha)}(x,\mathbf{y})$ by selecting $\alpha\in[0,1]$ and
a pair of indices $i,j$ (possibly the same) with
$\partial G_i(x,\mathbf{y})/\partial y_j\ \neq\ 0$ and then 
substituting $\alpha G_j(x,\mathbf{y}) + (1-\alpha) y_j$ for a single occurrence of $y_j$ in the power series $G_i$. 
Suppose $ H(x,y_0;\mathbf{y} )$ is the result of replacing the single occurrence of $y_j$ in $G_i$ by 
a new variable $y_0$. Then the system $\Sigma^{(\alpha)}$ is
\begin{equation*}
\Sigma^{(\alpha)}: \quad \left\{
\begin{array}{ r c l}
y_1& =&  G_1^{(\alpha)} (x,\mathbf{y})\ :=\ G_1(x,\mathbf{y})\\
&\vdots&\\
y_i&=& G_i^{(\alpha)}(x,\mathbf{y})\ :=\ H\big(x,\alpha G_j(x,\mathbf{y}) + (1-\alpha)y_j); \mathbf{y}\big)\\
&\vdots&\\
y_m& =& G_m^{(\alpha)} (x,\mathbf{y})\ :=\ G_m(x,\mathbf{y})
\end{array}\right.
\end{equation*}
More generally, a system $\Sigma^\star: \mathbf{y}=\bG^\star(x,\mathbf{y})$ is a \textit{self-substitution transform} of
$\Sigma: \mathbf{y}=\bG(x,\mathbf{y})$ if there is a sequence $\Sigma_0, \Sigma_1,\ldots,\Sigma_r$ of systems
such that $\Sigma = \Sigma_0$, $\Sigma^\star = \Sigma_r$, and for $0\le i <r$ the system 
$\Sigma_{i+1}$ is a minimal self-substitution transform of $\Sigma_i$. 

\begin{lemma}\label{subs triv}
For $\Sigma^{(\alpha)}$ and $\Sigma^\star$ as described above:
\begin{thlist}
\item
$\Sigma = \Sigma_0$.
\item
If $\Sigma$ is irreducible and $\alpha\in[0,1)$ then $\Sigma^{(\alpha)}$ is irreducible.
\item
Suppose $\Sigma$ is irreducible. Then $\Sigma^\star$ is irreducible iff each step $\Sigma_i$
is irreducible.
\item
Suppose $\Sigma$ is well-conditioned and $\alpha\in[0,1]$. 
Then $\Sigma^{(\alpha)}$ is
well-conditioned iff it is irreducible. In particular $\Sigma^{(\alpha)}$ is well-conditioned if $\alpha\in[0,1)$.
\item
Suppose $\Sigma$ is well-conditioned. 
Then $\Sigma^\star$ is well-conditioned iff it is irreducible.
\end{thlist}
\end{lemma}

\begin{proof}
Straightforward.
\end{proof}

\begin{lemma} \label{subs preserv}
Suppose \[\Sigma^\star: \mathbf{y} = \bG^\star(x,\mathbf{y})\] is a self-substitution transform of 
a well-conditioned $\Sigma: \mathbf{y} = \bG(x,\mathbf{y})$. 
Then the following hold:
\begin{thlist}
\item
$\bG(x,\mathbf{y})$ and $\bG^\star (x,\mathbf{y})$ have the same positive domain of convergence.
\item
 $\Sigma^\star$ and $\Sigma$ have the same positive solutions and the same characteristic points.
\item
If $\Sigma^\star$ is  well-conditioned then $\Sigma$ and $\Sigma^\star$ 
have the same standard solution $\bf{T}(x)$ and extreme point $(\rho,\pmb{\tau})$.
\item
If $\Sigma^\star$ is  well-conditioned then the Jacobians $J_\bG(x,\mathbf{y})$ and $J_{\bG^\star}(x,\mathbf{y})$
have all entries finite at the same positive points $(a,\mathbf{b})$ in the domain of $\bG$.
\end{thlist}
\end{lemma}

\begin{proof}
It suffices to prove this for the case that \[\Sigma^\star=\Sigma^{(\alpha)},\] a minimal self-substitution transform of 
$\Sigma$ as described above, namely substituting $\alpha G_j(x,\mathbf{y}) + (1-\alpha) y_j$ for a single occurrence 
of $y_j$ in the power series $G_i(x,\mathbf{y})$.  Let
$$
H(x,y_0; \mathbf{y}) = A(x,\mathbf{y}) y_0 + B(x,\mathbf{y}),
$$
where $A(x,\mathbf{y}) $ and $B(x,\mathbf{y}) $ are power series with non-negative coefficients, and neither is 0,
be such that
\begin{eqnarray*}
G_i(x,\mathbf{y})&=& A(x,\mathbf{y}) y_j + B(x,\mathbf{y})\\
G_i^{(\alpha)}(x,\mathbf{y})
&=& A(x,\mathbf{y})\big(\alpha G_j(x,\mathbf{y})+ (1-\alpha) y_j\big) + B(x,\mathbf{y}).
\end{eqnarray*}
For item (a), first suppose that $(a,\mathbf{b}) \in\domp(\bG)$. Then $A(a,\mathbf{b})$ and $B(a,\mathbf{b})$ are finite,
so $G_i^{(\alpha)}(a,\mathbf{b})$ is finite. This suffices to show $(a,\mathbf{b}) \in\domp(\bG^{(\alpha)})$ 
since the other $\bG_j^{(\alpha)}(x,\mathbf{y})$ are the same as those in $\Sigma$. Conversely, suppose 
$(a,\mathbf{b}) \in\domp(\bG^{(\alpha)})$. Again $A(a,\mathbf{b})$ and $B(a,\mathbf{b})$ are finite, so 
$G_i(a,\mathbf{b})$ is finite; and as before, the other $G_j(a,\mathbf{b})$ are finite. 
Thus $(a,\mathbf{b}) \in\domp(\bG)$.

For item (b),
if $i\neq j$ then clearly the two systems have the same positive solutions 
since $y_j = G_j(x,\mathbf{y})$ is in both systems. 

If $i=j$ first note that every positive solution of $\Sigma$ is also a solution of $\Sigma^{(\alpha)}$.
For the converse we have
\begin{eqnarray*}
G_i^{(\alpha)}(x,\mathbf{y})
&=& 
A(x,\mathbf{y})\Big(\alpha \big(A(x,\mathbf{y}) y_i + B(x,\mathbf{y})\big)+ (1-\alpha) y_i\Big) + B(x,\mathbf{y})\\
&=& \alpha A(x,\mathbf{y})^2 y_i + \alpha A(x,\mathbf{y}) B(x,\mathbf{y})
+ (1-\alpha) A(x,\mathbf{y})y_i + B(x,\mathbf{y}).
\end{eqnarray*}

Let $(a,\mathbf{b})$ be a positive solution of $\Sigma^{(\alpha)}$.
Then  $(a,\mathbf{b})$ 
solves all equations $y_j = G_j(x,\mathbf{y})$ of $\Sigma$ where $j\neq i$ since
these equations are also in $\Sigma^{(\alpha)}$. Now 
\begin{eqnarray*}
b_i\ &=& G_i^{(\alpha)}(a,\mathbf{b})\\
&=&\ \alpha A(a,\mathbf{b})^2 b_i + \alpha A(a,\mathbf{b}) B(a,\mathbf{b})
+ (1-\alpha) A(a,\mathbf{b})b_i + B(a,\mathbf{b}),
\end{eqnarray*}
so
$$
\Big(1 - \alpha A(a,\mathbf{b})^2 - (1-\alpha) A(a,\mathbf{b})\Big)b_i\ =\ \Big(1+\alpha A(a,\mathbf{b})\Big) B(a,\mathbf{b}).
$$
Since $1+\alpha A(a,\mathbf{b})$  is positive, one can cancel  to obtain
$$
b_i\ =\ A(a,\mathbf{b}) b_i + B(a,\mathbf{b}),
$$
which says that $(a,\mathbf{b})$ satisfies the $i$th equation of $\Sigma$, and thus all the
equations of $\Sigma$. Consequently $\Sigma$ and $\Sigma^{(\alpha)}$ have the same positive 
solutions $(a,\mathbf{b})$.

To show  both systems have the same characteristic points, compute 
\begin{eqnarray}
\frac{\partial  G_i^{(\alpha)}(x,\mathbf{y} )}{\partial y_k} 
&=&  \frac{\partial  G_i(x,\mathbf{y} )}{\partial y_k}
\ +\ \alpha\frac{\partial  A(x,\mathbf{y} )}{\partial y_k} \cdot \big(G_j(x,\mathbf{y}) -y_j\big)\nonumber\\
&&\ +\ 
\alpha A(x,\mathbf{y} )\cdot \Big( \frac{\partial G_j(x,\mathbf{y})}{\partial y_k}- \delta_{jk}\Big). \label{JJSig}
\end{eqnarray}
At a positive solution $(a, \mathbf{b})$ to $\Sigma$ (hence to $\Sigma^\star$), this gives
\begin{eqnarray*}
\frac{\partial  G_i^{(\alpha)}(a,\mathbf{b} )}{\partial y_k} 
&=&  \frac{\partial  G_i(a,\mathbf{b} )}{\partial y_k}
\ +\ 
\alpha A(a,\mathbf{b} )\cdot \Big( \frac{\partial G_j(a,\mathbf{b})}{\partial y_k}- \delta_{jk}\Big).
\label{J to JSigma}
\end{eqnarray*}
Thus, since $(a,\mathbf{b})$ is positive,  one obtains 
$ J_\alpha(a,\mathbf{b}) :=  I -  J_{\bG^{(\alpha)}}(a,\mathbf{b})$ from   
$ J(a,\mathbf{b}) :=  I -  J_{\bG}(a,\mathbf{b})$ by an elementary row operation.
It follows that $\det( J(a,\mathbf{b}))=0$ if and only if $\det( J_\alpha(a,\mathbf{b}))=0$. 
Combining this with the fact that $\Sigma$ and $\Sigma^{(\alpha)}$ 
have the same positive solutions shows that they also have the same characteristic points.
%

For a well-conditioned system $\Sigma$, the standard solution is the unique sequence 
$\mathbf{T}(x)$ of non-negative
power series with $\mathbf{T}(0) = \mathbf{0}$ that solve the system. The standard solution of $\Sigma$ is
clearly a solution of $\Sigma^{(\alpha)}$. Thus if $\Sigma^{(\alpha)}$ is well-conditioned then it has the same standard
solution, and hence the same extreme point, as $\Sigma$, so (c) holds.

For the final item, let $(a,\mathbf{b})$ be a point in $\domp(\bG)$, hence a point in $\domp(\bG^{(\alpha)})$.
$A(a,\mathbf{b})$ is finite by looking at the expression above for $\bG_i(x,\mathbf{y})$. Then, 
since $G_j^{(\alpha)}(x,\mathbf{y}) = G_j(x,\mathbf{y})$ for $j\neq i$, 
\eqref{JJSig} shows that $\dfrac{\partial G_i^{(\alpha)}(a,\mathbf{b} )}{\partial y_k} $ is finite iff
$\dfrac{\partial G_i(a,\mathbf{b} )}{\partial y_k} $ is finite, 
so one has item (d).

\end{proof}

\begin{lemma} \label{subs lemma}
A well-conditioned system $\Sigma: \mathbf{y}=\bG(x,\mathbf{y})$ can be transformed by a 
self-substitution into  a well-conditioned system $\Sigma^\star: \mathbf{y}=\bG^\star(x,\mathbf{y})$ 
such that the Jacobian
matrix $J_{\bG^\star}(x,\mathbf{y})$ has all entries non-zero. Indeed, given any $n>0$, one
can find a $\Sigma^\star$ such that all $n$th partials of the $G_i^\star$ with respect to the $y_j$
 are non-zero.
\end{lemma}

\begin{proof}
The goal is to show that there is a sequence $\Sigma_0,\ldots, \Sigma_r$ of minimal self-substitution transforms
that go from
$\Sigma$ to the desired $\Sigma^\star$, and such that each system $\Sigma_i$ is well-conditioned.
The following four cases give the key steps in the proof.
\medskip

\noindent
{\sf CASE I:} Suppose some $G_i$ is such that all $n$th partials are non-zero. If 
$G_j $  is dependent on $y_i$ (there is at least one such $j$)
then substituting
$(1/2)G_{i} + (1/2)y_i$ for some occurrence of $y_i$ in $G_{j}$ gives a  well-conditioned system 
$\Sigma'$ such that for $G_i ' = G_i$ and $G_{j} '$, all  $n$th partials
are non-zero. Continuing in this fashion one eventually has  the desired system $\Sigma^\star$.

\medskip

\noindent
{\sf CASE II:} Suppose $\dfrac{\partial^{mn} G_i }{ \partial {y_i}^{mn} }\  \neq \ 0$ for some $i$. 
This means $ {y_i}^{mn}$ divides some monomial of $G_i$. Use the fact that for any $j\neq i$ 
there is a dependency path from $y_i$ to $y_j$ to convert, via self-substitutions that preserve
the well-conditioned property, a product of $n$ of 
the $y_i$ in this monomial into a power series which has ${y_j}^n$ dividing one of its monomials. By doing
this for each $j\neq i$ one obtains  a well-conditioned $G_i '$ with
$$
\frac{\partial^{mn} G_i '}{\partial {y_1}^n \cdots \partial {y_m}^n}\ \neq \ 0.
$$
$\Sigma'$ is now in Case I.
\medskip

\noindent
{\sf CASE III:} Suppose $\dfrac{\partial^2 G_i }{ \partial {y_i}^2 }\  \neq \ 0$ for some $i$. Substituting 
$G_i$ for a  suitable occurrence of $y_i$ in $G_i$ gives a well-conditioned 
$\Sigma'$
where 
$\dfrac{\partial^3 G_i '}{ \partial {y_i}^3 }\  \neq \ 0$. 
Continuing in this fashion leads to Case II.
\medskip

\noindent
{\sf CASE IV:} Suppose $\dfrac{\partial^2 G_i }{ \partial {y_j}\partial{y_k} }\  \neq \ 0$ for some $i,j,k$. 
If $j\neq i$ there is a dependency path from $y_j$ to $y_i$ which shows how to make self-substitutions (that preserve
the well-conditioned property) leading to 
$\dfrac{\partial^2 G_i }{ \partial {y_i}\partial{y_k} }\  \neq \ 0$.  
Likewise, if $k\neq i$ there is a dependency path from $y_k$ to $y_i$ which shows how to make self-substitutions (with each
minimal step being well-conditioned) leading to 
$\dfrac{\partial^2 G_i }{ \partial {y_i}^2}\  \neq \ 0$, which is Case III.
\medskip

Since $\Sigma$ is non-linear in $\mathbf{y}$, for some $i,j,k$ we have 
$$\frac{\partial^2 G_i }{ \partial {y_i} \partial y_k}\  \neq \ 0.$$
Thus starting with Case IV and working back to Case I we arrive at the desired $\Sigma^\star$. 

\end{proof}

\begin{lemma} \label{Lambda=1}
Let $\Sigma: \mathbf{y} = \bG(x,\mathbf{y})$ be a well-conditioned system and let 
$\Sigma^\star: \mathbf{y} = \bG^\star(x, \mathbf{y})$ be  a  
self-substitution transform of $\Sigma$ that is also well-conditioned.
If $(a,\mathbf{b})$ is a characteristic point of $\Sigma$, hence of $\Sigma^\star$, then 
$\Lambda(a,\mathbf{b}) = 1$ iff $\Lambda^\star(a, \mathbf{b})=1$.  
\end{lemma}

\begin{proof}
Let $(a,\mathbf{b})$ be a characteristic point of $\Sigma$.
  It suffices to consider the case where $\Sigma^\star$ is obtained from $\Sigma$ by a 
  minimal self-substitution.  Let $G_i(x,\mathbf{y})$ depend on $y_j$,  and let 
  $H(x,y_0;\mathbf{y})$ be the result of replacing a single occurrence of $y_j$ in 
  $G_i(x,\mathbf{y})$ by $y_0$. Then let $\Sigma^{(\alpha)} : \mathbf{y} = \bG^{(\alpha)}(x,\mathbf{y})$,
  $\alpha \in [0,1]$,
  be the minimal
  self-substitution transform of $\Sigma$ obtained by applying
  the substitution $y_0\gets \alpha G_j(x,\mathbf{y}) + (1-\alpha)y_j$ to $H(x,y_0;\mathbf{y})$
to obtain
  $$
  G_i^{(\alpha)}(x,\mathbf{y})\:=\ H\big(x,\alpha G_j(x,\mathbf{y}) + (1-\alpha)y_j); \mathbf{y}\big).
  $$ 
  Let $\Lambda_\alpha :=\Lambda_\alpha(a,\mathbf{b})$, the largest real eigenvalue of 
  $J_{\bG^{(\alpha)}}(a,\mathbf{b})$.
   
  The only information that we need from the above construction of the $\bG_i^{(\alpha)}$
   is that the function 
  $\alpha \mapsto J_{\bG^{(\alpha)}}(a,\mathbf{b})$ is continuous on $[0,1]$, and each
  $J_{\bG^{(\alpha)}}(a,\mathbf{b})$ has 1 being an eigenvalue.
  Since $\Lambda$ is continuous on non-negative matrices by Corollary \ref{prop of Lambda}, 
  it follows that $\alpha\mapsto\Lambda_\alpha$ 
  is continuous on $[0,1]$.
  The goal is to show that one has $\Lambda_0=1$ iff $\Lambda_\alpha=1$.
  
  Since $(a,\mathbf{b})$ is a characteristic point of $\Sigma_0$ it is also a characteristic point of $\Sigma^{(\alpha)}$,
  by Lemma \ref{subs preserv}, for $\alpha\in[0,1]$. Thus 1 is an eigenvalue of $J_{\bG^{(\alpha)}}(a,\mathbf{b})$ for 
  $\alpha\in[0,1]$. Suppose $\Lambda_0= 1$. Suppose there is a $\beta\in(0,1]$ with
  $\Lambda_\beta  > 1$. From the continuity of $\Lambda_\alpha$ there is a $\gamma\in[0,\beta)$ such that: 
  $\Lambda_\gamma = 1$, and $\Lambda_\alpha>1$ for $\alpha\in(\gamma,\beta]$.
  
  Let $p_\alpha(x)$ be the characteristic polynomial of $J_{\bG^{(\alpha)}}(a,\mathbf{b})$. From
 $$
  p_\alpha(1)\ =\ p_\alpha(\Lambda_\alpha)\ =\ 0
$$
  one has, for each $\alpha\in(\gamma,\beta)$, a $c_\alpha\in(1,\Lambda_\alpha)$ such that
  $$
  \frac{d p_\alpha}{d x}(c_\alpha)\ =\ 0.
  $$
Since $\Lambda_\alpha$ is continuous on $[0,1]$,
  \[
  \lim_{\alpha\rightarrow \gamma^+} \Lambda_\alpha\ =\ \Lambda_\gamma\ =\ 1.
  \]
 This implies
   $
  \lim_{\alpha\rightarrow \gamma^+} c_\alpha\ =\ 1,
  $
  and thus
   $$
  \frac{d p_\gamma}{d x}(1)\ =\ \lim_{\alpha\rightarrow \gamma^+} \frac{d p_\alpha}{d x}(c_\alpha) \ =\ 0.
  $$
  But from the Perron-Frobenius theory (see Proposition \ref{PF Prop})
   we know that $\Lambda_\gamma = 1$ implies that 1 is a simple root
  of $p_\gamma(x)$, giving a contradiction. Thus $\Lambda_0=1$ implies $\Lambda_\alpha=1$.
  
  A similar proof gives the converse, that if $\Lambda_\alpha=1$  then $\Lambda_0=1$,
  proving the lemma. 
  
 \end{proof}

\begin{remark}\label{nonzero partials}
In view of the last two lemmas, given a well-conditioned system $\Sigma:\mathbf{y}=\bG(x,\mathbf{y})$,
 when one wants to prove something about the positive solutions, the  characteristic points, 
 or whether or not $\Lambda(a,\mathbf{b}) = 1$ at a characteristic point $(a,\mathbf{b})$,
one can, given any $n>0$, assume 
without loss of generality that all $n$th partials of each $G_i$ with respect to the $y_j$ are non-zero.
 In the rather scant literature on nonlinear systems one finds a preference for working with aperiodic systems
 (see, e.g., \cite{FlSe2009}), no doubt because of the simplicity of using uniform substitutions to convert such a system into one where the Jacobian 
 of $\bG$ has non-zero entries. 
With Lemmas \ref{subs lemma} and \ref{Lambda=1}, the need for the aperiodic hypothesis is avoided. 
\end{remark}


\subsection{Basic Properties of $(\rho,\pmb{\tau})$ and $\mathcal{CP}$}

Now we turn to the question of how to find information about the extreme point 
$(\rho,\pmb{\tau})$ of a well-conditioned system $\Sigma$ without solving
the system for the standard solution $\bT(x)$.

\begin{lemma} \label{J entries}
Let $\mathbf{y} = \bG(x,\mathbf{y})$ be a well-conditioned system with all entries of $J_\bG$ non-zero.
\begin{thlist}
\item
One has the formal equality
\begin{equation}\label{diff eq}
\bT'(x)\ =\  \bG_x\big(x,\bT(x)\big) + J_\bG\big(x,\bT(x)\big)\cdot\bT'(x),
\end{equation}
which also holds for $x\in [0,\infty]$.
\item
All $T_i'(\rho)$ are finite or all $T_i'(\rho)=\infty$.
\item
For all $i,j$ the following hold:
\begin{eqnarray*}
0& <& \frac{\partial G_i}{\partial y_j}(\rho,\pmb{\tau})\cdot\frac{\partial G_j}{\partial y_i}(\rho,\pmb{\tau})\ \le \ 1\\
0& < & \frac{\partial G_i}{\partial y_j}(\rho,\pmb{\tau})\ <\ \infty\\
0& < & \frac{\partial G_i}{\partial y_i}(\rho,\pmb{\tau})\ \le\ 1.
\end{eqnarray*}
\end{thlist}
\end{lemma}

\begin{proof}
Differentiating \eqref{the sol} gives \eqref{diff eq}, so $\bT'(x)$
is a solution to the irreducible system
$\mathbf{u}\ =\  \bG_x\big(x,\bT(x)\big) + J_\bG\big(x,\bT(x)\big)\cdot \mathbf{u}$, implying (b).
For $x\in(0,\rho)$, for each $i,j$, \eqref{diff eq} implies
$$ 
T'_i(x)\ >\ \frac{\partial G_i}{\partial y_j}\big(x,\bT(x)\big)\cdot T_j'(x),
$$
and thus
\begin{eqnarray*}
1& >& \frac{\partial G_i}{\partial y_j}\big(x,\bT(x)\big)\cdot\frac{\partial G_j}{\partial y_i}\big(x,\bT(x)\big)\ >\ 0,
\end{eqnarray*}
giving the inequalities in (c) since the value of $ \dfrac{\partial G_i}{\partial y_j}\big(\rho,\pmb{\tau}\big)$ is
the limit of $\dfrac{\partial G_i}{\partial y_j}\big(x,\bT(x)\big)$ as $x$ approaches 
$\rho$ from below.
\end{proof}

\begin{lemma} \label{easy CP}
Let $\mathbf{y} = \bG(x,\mathbf{y})$ be a well-conditioned system.
\begin{thlist}

\item 
If $(a,\mathbf{b})\in\mathcal{CP}$ then $\Lambda(a,\mathbf{b}) \ge 1$.

\item
$0 < \Lambda\big(a,\bT(a)\big) < 1$, for $0<a<\rho$.

\end{thlist}
\end{lemma}

\begin{proof}
For (a) note that $(a,\mathbf{b})\in\mathcal{CP}$ implies that 1 is an eigenvalue of 
$J_\bG(a,\mathbf{b})$, so $\Lambda(a,\mathbf{b}) \ge 1$.

(b) Given $0<a<\rho$, by the Perron-Frobenius theory of nonnegative matrices we know that
there is a positive left eigenvector (a row vector) $\bfv$ belonging to $\Lambda\big(a,\bT(a)\big)$. By \eqref{diff eq}
$$
\bfv\cdot\bT'(a)\ =\ \bfv\cdot \bG_x\big(a,\bT(a)\big) + \bfv\cdot J_\bG\big(a,\bT(a)\big)\cdot\bT'(a),
$$
so
$$
\bfv\cdot\bT'(a)\ =\ \bfv\cdot \bG_x\big(a,\bT(a)\big) + \Lambda\big(a,\bT(a)\big) \bfv\cdot \bT'(a).
$$
Since $\bfv\cdot \bT'(a) > 0 $ and $\bfv\cdot \bG_x\big(a,\bT(a)\big) > 0$ it follows that 
$\Lambda\big(a,\bT(a)\big) < 1$. 

\end{proof}

\begin{proposition}\label{antichain}
Let $\mathbf{y} = \bG(x,\mathbf{y})$ be a well-conditioned system.
Suppose $(a,\mathbf{b})$ and $(c,\mathbf{d})$ are characteristic points and $(a,\mathbf{b}) \le (c,\mathbf{d})$.
Then
$(a,\mathbf{b}) = (c,\mathbf{d})$. Thus the set of  characteristic points of the system forms an antichain under the partial ordering $\le$.
\end{proposition}

\begin{proof}
For the proof assume, in view of Remark \ref{nonzero partials}, that all second partials of the $G_i$ with respect to the $y_j$ do not vanish.
If $\mathbf{b}=\mathbf{d}$ then $\bG(a,\mathbf{b}) = \mathbf{b} = \mathbf{d} = \bG(c,\mathbf{d})$, 
which forces $a=c$ by the monotonicity of each $\bG_i$. 

Now assume $\mathbf{b}\neq \mathbf{d}$. Since $\mathbf{b}\le \mathbf{d}$, all
entries of $\mathbf{d} - \mathbf{b}$ are non-negative.
Using part of a Taylor series expansion,
$$
\bG(c,\mathbf{d})\ \ge\ \bG(a,\mathbf{b}) +  J_\bG(a,\mathbf{b})
(\mathbf{d} - \mathbf{b})
+ \frac{1}{2}
\left[ 
\begin{array}{c}
\frac{\partial^2 G_1(a,\mathbf{b})}{\partial {y_1}^2}(d_1 - b_1)^2\\
\vdots\\
\frac{\partial^2 G_m(a,\mathbf{b})}{\partial {y_m}^2}(d_m - b_m)^2
\end{array}
\right].
$$
Since $\bG(a,\mathbf{b}) = \mathbf{b}$ and  $\bG(c,\mathbf{d})= \mathbf{d} $,
$$
\mathbf{d} - \mathbf{b}\ \ge\   J_\bG(a,\mathbf{b}) (\mathbf{d} - \mathbf{b})
+ \frac{1}{2}
\left[ 
\begin{array}{c}
\frac{\partial^2 G_1(a,\mathbf{b})}{\partial {y_1}^2}(d_1 - b_1)^2\\
\vdots\\
\frac{\partial^2 G_m(a,\mathbf{b})}{\partial {y_m}^2}(d_m - b_m)^2
\end{array}
\right].
$$
Let $\lambda$ be the largest real eigenvalue of the positive matrix $J_\bG(a,\mathbf{b})$, 
and let $\bfv$ be a positive left eigenvector belonging to $\lambda$. Then
\begin{eqnarray*}
\bfv(\mathbf{d} - \mathbf{b})
&\ge&   \bfv J_\bG(a,\mathbf{b}) (\mathbf{d} - \mathbf{b})
+ \frac{1}{2}\bfv \left[ 
\begin{array}{c}
\frac{\partial^2 G_1(a,\mathbf{b})}{\partial {y_1}^2}(d_1 - b_1)^2\\
\vdots\\
\frac{\partial^2 G_m(a,\mathbf{b})}{\partial {y_m}^2}(d_m - b_m)^2
\end{array}
\right]\\
&=&
\lambda \bfv (\mathbf{d} - \mathbf{b}) + \frac{1}{2}\bfv 
\left[ 
\begin{array}{c}
\frac{\partial^2 G_1(a,\mathbf{b})}{\partial {y_1}^2}(d_1 - b_1)^2\\
\vdots\\
\frac{\partial^2 G_m(a,\mathbf{b})}{\partial {y_m}^2}(d_m - b_m)^2
\end{array}
\right]
\end{eqnarray*}
so
\begin{eqnarray*}
(1-\lambda)\bfv(\mathbf{d} - \mathbf{b})
&\ge&
 \frac{1}{2}\bfv \left[ 
\begin{array}{c}
\frac{\partial^2 G_1(a,\mathbf{b})}{\partial {y_1}^2}(d_1 - b_1)^2\\
\vdots\\
\frac{\partial^2 G_m(a,\mathbf{b})}{\partial {y_m}^2}(d_m - b_m)^2
\end{array}
\right]\ >\ 0,
\end{eqnarray*}
and this forces $\lambda < 1$, contradicting Lemma \ref{easy CP} (a).
\end{proof}

\begin{lemma}\label{ev of J(x,Tx)} 
Let $\mathbf{y} = \bG(x,\mathbf{y})$ be a well-conditioned system.
\begin{thlist}

\item
$(\rho,\pmb{\tau})$ is in the domain of $J_\bG(x,\mathbf{y})$, that is, all entries of the 
matrix $J_\bG(\rho,\pmb{\tau})$ are finite.

\item
If $(\rho,\pmb{\tau})$ is in the interior of the domain of $\bG(x,\mathbf{y})$ then it is a 
characteristic point.

\item
$0<\Lambda(\rho,\pmb{\tau}) \le 1$.

\item
$\Lambda(\rho,\pmb{\tau}) = 1$ iff $1$ is an eigenvalue of 
$J_\bG(\rho,\pmb{\tau})$ iff $(\rho,\pmb{\tau})\in\mathcal{CP}$.

\end{thlist}
\end{lemma}

\begin{proof}
For item (a), first let $\Sigma^\star$ be a well-conditioned self-substitution 
transform of $\Sigma$ with all entries in $J_{\bG^\star}(x,\mathbf{y})$ non-zero (see Remark \ref{nonzero partials}).
By Lemma \ref{J entries}, all entries of $J_{\bG^\star}(\rho,\pmb{\tau})$ 
are finite. Then Lemma \ref{subs preserv}(d) shows that all entries of 
$J_{\bG}(\rho,\pmb{\tau})$ are finite.

For the remainder of the proof  
we can assume that all entries in $J_\bG$ are non-zero.  
For part (b) one argues just as in the case of a single equation---if $(\rho,\pmb{\tau})$
is an interior point but not a characteristic point then by the implicit function theorem
there would be an analytic continuation of $\bT(x)$ at $\rho$, which is impossible.

For (c), since $\Lambda$  is a continuous nondecreasing function by Corollary \ref{prop of Lambda}, and since the limit of 
$J_\bG\big(x,\bT(x)\big)$ as $x$ approaches $\rho$ from below is $J_\bG(\rho,\pmb{\tau})$, it
follows from Lemma \ref{easy CP} (b) that
$\Lambda(\rho,\pmb{\tau})\le 1.$

For (d), clearly $\Lambda(\rho,\pmb{\tau}) = 1$ implies
$1$ is an eigenvalue of $J_\bG(\rho,\pmb{\tau})$, and this in turn implies that
$(\rho,\pmb{\tau})\in\mathcal{CP}$.  Now suppose that $(\rho,\pmb{\tau}) \in \mathcal{CP}$.
Then 1 is an eigenvalue of 
$J_\bG(\rho,\pmb{\tau})$, so $\Lambda(\rho,\pmb{\tau})\ge 1.$
Thus (c) gives $\Lambda(\rho,\pmb{\tau}) = 1.$ 

\end{proof}

\begin{lemma}\label{structure of CP}
Let $\mathbf{y} = \bG(x,\mathbf{y})$ be a well-conditioned system.
If $(a,\mathbf{b})$ is a characteristic point and 
$(a,\mathbf{b})\neq (\rho,\pmb{\tau})$ then either
\begin{thlist}
\item
$b_i > \tau_i$ for all $i$, or 
\item
$a < \rho$ and $b_i > T_i(a)$ for all $i$, and some $b_j>\tau_j$.
\end{thlist}

\end{lemma}

\begin{proof}
Conditions (c) and (d) in the definition of well-conditioned ensures that each $G_i(x,\mathbf{y})$ 
depends on $x$.
In view of Remark \ref{nonzero partials}~assume that all second partials of each $G_i(x,\mathbf{y})$ with respect to the ${y_j}$ are non-zero. 
 Suppose that (a) does not hold.

\medskip
\noindent
{\bf Claim 1:} 
If  some $b_i >  \tau_i$ and some $b_j \le  \tau_j$ then 
$a < \rho$ and $ T_i(a) < b_i$ for $1\le i\le m$.

WLOG assume that $$b_1\le  \tau_1,\ldots, b_k\le  \tau_k$$ and 
$$b_{k+1} > \tau_{k+1},\ldots,b_m >  \tau_m.$$

{}From  the monotonicity and continuity of the $ T_i$ on $[0,\rho]$ it follows that for $1\le i\le k$ there exist unique 
$\xi_i\in (0,\rho]$ such that $$b_i= T_i(\xi_i).$$ WLOG assume that
$$0<\xi_1\le \cdots \le \xi_k\le \rho.$$
 For  $i\in\{1,\ldots,k\}$ 
 $$ T_i(\xi_1)\ \le\  T_i(\xi_i)\ =\  b_i$$
and for $k+1\le i\le m$
 $$ T_i(\xi_1)\  \le\  T_i(\rho)\ <  b_i.$$

Now suppose $\xi_1 < a$. Then
\begin{eqnarray*}
b_1 
&=& G_1\big(\xi_1, T_1(\xi_1),\ldots, T_m(\xi_1)\big)\\
&<& G_1 (a,b_1,\ldots,b_m)\ =\ b_1,
\end{eqnarray*}
a contradiction. 
Thus   $$0< a \le \xi_1\le \cdots \le \xi_k\le \rho.$$ 

Using this one has, for $1\le i \le k$:
\begin{eqnarray*}
 T_i(\xi_i)
&=& G_i \big(a, T_1(\xi_1),\ldots, T_k(\xi_k),b_{k+1},\ldots,b_m\big)\\
&>& G_i\big(a, T_1(a),\ldots, T_k(a), T_{k+1}(a),\ldots, T_m(a)\big)
\ =\  T_i(a).
\end{eqnarray*}
Thus   for $1\le i \le k$, 
$$
\begin{array}{l @{\quad} l}
 0 < a < \xi_i\le \rho\\
  T_i(a) < T_i(\xi_i) = b_i .
\end{array}
$$
Furthermore, for $k+1\le i\le m$,
$$
 T_i(a)< T_i(\rho) < b_i.
$$
Thus, in this case, for $1\le i \le m$ one has $T_i(a) < b_i$.
\medskip

\noindent
{\bf Claim 2:}
If $b_i \le  T_i(\rho)$ for all $i$ then $a < \rho$ and $ b_i= T_i(a)$ for all $i$.
\smallskip

Choose $\xi_i \in (0,\rho]$ such that $b_i =  T_i(\xi_i)$. 
WLOG one can assume $0 < \xi_1\le \cdots \le \xi_m \le \rho$. If $\xi_1 < a$
then
\begin{eqnarray*}
b_1
&=&G_1\big(a, T_1(\xi_1),\ldots, T_m(\xi_m)\big)\\
&>&G_1\big(\xi_1, T_1(\xi_1),\ldots, T_m(\xi_1)\big)\\
&=& T_1(\xi_1)\ =\ b_1,
\end{eqnarray*}
a contradiction.  Thus $a \le \xi_1\le\cdots\le\xi_m \le \rho$.

Next one has
\begin{eqnarray*}
b_m
&=&G_m\big(\xi_m, T_1(\xi_m),\ldots, T_m(\xi_m)\big)\\
&\ge&G_m\big(a, T_1(\xi_1),\ldots, T_m(\xi_m)\big)\\
&=& b_m,
\end{eqnarray*}
so the $\ge$ step must be an equality, and this implies $\xi_m=a$. Thus
all $\xi_i = a$, and then for all $i$ one has $b_i =  T_i(a)$. Since
  $(a,\mathbf{b}) = (a,\mathbf{T}(a))$ is assumed to be a different 
  characteristic point from
$(\rho,\pmb{\tau})$, it follows that $a<\rho$.
\medskip

\noindent
{\bf Claim 3:}
It is not the case that
$b_i \le  \tau_i$ for all $i$.
\smallskip

Otherwise by Claim 2 we would have $(a,\mathbf{b}) = (a,\mathbf{T}(a))$ with 
$0<a<\rho$, and then by Lemma \ref{easy CP}  it would follow that
$(a,\mathbf{b})\notin\mathcal{CP}$.
But by assumption, $(a,\mathbf{b}) \in\mathcal{CP}$.

\end{proof}


\begin{theorem}\label{location}
Suppose $(\rho,\pmb{\tau})$ is a characteristic point of
a well-conditioned system $\mathbf{y} = \bG(x,\mathbf{y})$.
 Then:
\begin{thlist}
\item
$\rho$ is the largest first coordinate of any characteristic point, that is
$$
\rho \ =\ \max\Big\{a : (a,\mathbf{b})\in\mathcal{CP}\Big\},
$$
\item
$(\rho,\pmb{\tau})$ is the only characteristic point whose first coordinate is  $\rho$.
\end{thlist}
\end{theorem}

\begin{proof}
Use  Proposition \ref{antichain} and Lemma \ref{structure of CP}.
\end{proof}

Turning to 1-equation systems, we have the following results.

\begin{proposition}\label{1 eq CP}
A well-conditioned 1-equation system $y=G(x,y)$ has a most one characteristic point;
if there is such a point it must be the extreme point $(\rho,\tau)$ of the standard solution $T(x)$.
\end{proposition}
\begin{proof}
The characteristic system is
\begin{eqnarray*}
y&=&G(x,y)\\
1&=&G_y(x,y).
\end{eqnarray*}
Suppose $(a,b)\in\mathcal{CP}$ is different from $(\rho,\tau)$. 
Then $b>\tau$ by Lemma \ref{structure of CP}.

CASE 1: Suppose $a >\rho$. Then $(\rho,\tau)$ is in the interior of
$\domp(G)$, so $(\rho,\tau)\in\mathcal{CP}$ by Lemma \ref{ev of J(x,Tx)}(b).
But this violates the antichain condition of Proposition \ref{antichain} for $\mathcal{CP}$.

CASE 2: Suppose $a\le\rho$. Then $b=G(a,b)$ and $T(a) = G(a,T(a))$ leads to
$1 = G_y(a,\xi)$ for some $T(a) < \xi < b$. 
 But $G_y(a,b)=1$ since $(a,b)\in\mathcal{CP}$, so again we
have a contradiction by the strict monotonicity of $G_y(x,y)$ in $\domp(G)$.

Thus the only possible $(a,b)\in\mathcal{CP}$ is $(\rho,\tau)$.
\end{proof}

\begin{remark}Meir and Moon \cite{MeMo1989} prove that well-conditioned
1-equation systems have at
most one characteristic point in the {interior} of $\domp(G)$; and if such a point exists
then it must be $(\rho,\tau)$. See also
Flajolet and Sedgewick \cite{FlSe2009}, Chapter VII $\S$4.
\end{remark}

The {\em simple} 1-equation systems $y = xA(y)$ studied by Meir and Moon appear frequently in
the book \cite{FlSe2009} of Flajolet and Sedgewick. Letting $\rho_A$ be the radius of convergence of $A(y)$,
they use the hypothesis
\begin{equation} \label{hypMM}
\lim_{y\rightarrow {\rho_A}^-} \frac{yA'(y)}{A(y)} \ >\ 1
\end{equation}
to guarantee that $(\rho,\tau)$ is in the interior of the domain of convergence of $xA(y)$.
The following corollary improves on their results by giving a precise condition for there to be a characteristic point (which must be $(\rho,\tau)$ by Proposition \ref{1 eq CP}), and giving a precise condition for when $(\rho,\tau)$ is
a characteristic point on the boundary [in the interior] of $\domp(G)$.

\begin{corollary} \label{simple sys}
Suppose $y=G(x,y)$ is a well-conditioned 1-equation system with \[G(x,y) = xA(y),\] that is,
$A(y)$ is a power series $\sum_{n\ge 0} a_n y^n$ 
 with non-negative coefficients, and both $A(0)$ and $A''(y)$ are non-zero.
Let $B(y) = yA'(y) - A(y) + A(0)$. Then the characteristic system is equivalent to
\begin{eqnarray*}
B(y) &=& A(0)\\
x&=& \frac{y}{A(y)},
\end{eqnarray*}
and,  one has
\begin{thlist}
\item
$\mathcal{CP} = \textrm{\O}$ iff $B(\rho_A) < A(0)$
\item
$B(\rho_A) \ge A(0)$ implies $\mathcal{CP} = \{(\rho,\tau)\}$
\item
$B(\rho_A) = A(0)$ implies $(\rho,\tau)$ is on the boundary of $\domp(G)$
\item
$B(\rho_A) > A(0)$ implies $(\rho,\tau)$ is in the interior of $\domp(G)$.
\end{thlist}
\end{corollary}

\begin{proof}
It is easy to verify the alternative form of the characteristic equations given in the corollary, and then note that
$$
B(y) \ =\ \sum_{n\ge 2} (n-1)a_ny^n
$$
is strictly increasing on $[0,\rho_A]$.
\end{proof}

\begin{remark} In Proposition VI.5 of \cite{FlSe2009} on simple 1-equation systems, 
the full well-conditioned hypothesis is not used, but
instead the non-linearity condition $A''(y) \neq 0$ is replaced by the stronger condition \eqref{hypMM}. 
This implies $B(\rho_A) > A(0)$, and thus one has $(\rho,\tau)$ in
the interior of $\domp(\bG)$.

 In the sentence following this proposition it is claimed that replacing \eqref{hypMM} by 
 $\rho_A = \infty$ gives hypotheses which imply \eqref{hypMM}. This is not correct
unless one adds in the condition $A''(y)\neq 0$, that is, the correct formulation is:
 {\em well-conditioned} plus $\rho_A=\infty$
implies \eqref{hypMM}. 
\end{remark}

\section{Eigenpoints}\label{sec eigenpoints}

The results developed so far do not give a practical way of locating $(\rho,\pmb{\tau})$
for well-conditioned systems with more than one equation. 
Even if one is successful in finding all the characteristic
points, no means has yet been formulated to determine 
if $(\rho,\pmb{\tau})$ is among them. In this
section special characteristic points called eigenpoints are shown to
provide the correct analog of characteristic points when moving from
1-equation systems to multi-equation systems.

\begin{proposition} \label{Lambda lemma}
Suppose  $(a,\mathbf{b})$ is a characteristic point of 
the well-conditioned system $\mathbf{y} = \bG(x,\mathbf{y})$.
 Then
 $\Lambda\big(a,\mathbf{b}\big)=1$ iff
 $(a,\mathbf{b}) = (\rho,\pmb{\tau})$.
\end{proposition}

\begin{proof}
 We can assume that no partial $\partial G_i/\partial y_j$ is zero.
 The direction $(\Leftarrow)$ follows from Lemma \ref{ev of J(x,Tx)} (d).
To prove the direction $(\Rightarrow)$ assume $(a,\mathbf{b}) \neq (\rho,\pmb{\tau})$.
By Lemma \ref{structure of CP} one has two cases to consider:
\begin{thlist}
\item[I]
 $a > \rho$ and for all $i$, $b_i > \tau_i$
 \item[II]
$a \le \rho$ and for all $i$, $b_i > T_i(a)$.
\end{thlist}

For  (I),  $(\rho,\pmb{\tau})$ is in the interior of the domain of $\bG$,
so by Lemma \ref{ev of J(x,Tx)} (b) it is a characteristic point. However this contradicts
Proposition \ref{antichain} which says the characteristic points form an antichain. 

For (II),
{}from the equations 
\begin{eqnarray*}
\bG(a,\mathbf{b}) - \mathbf{b}&=&\bf0\\
\bG\big(a,\bT(a)\big)\ - \bT(a)&=&\bf0
\end{eqnarray*}
one can apply a multivariate version of the mean value theorem to derive:
\begin{equation}\label{1 is ev}
\left( \frac{\partial G_i}{\partial y_j}(a,\bfv_{ij})\right)\big(\mathbf{b} - \bT(a)\big)\ =\ \mathbf{b} - \bT(a)
\end{equation}
with $\bfv_{ij} = \big(v_{ij}(1),\ldots,v_{ij}(m)\big)$ satisfying
$$
\begin{cases}
v_{ij}(r)= T_j(a)&\text{if }r >j\\
T_i(a) < v_{ij}(r) < b_i&\text{if }r =j\\
v_{ij}(r)= b_j&\text{if }r <j.
\end{cases}
$$
Clearly \eqref{1 is ev} shows that $\lambda = 1$ is an eigenvalue of 
$\left(\dfrac{\partial G_i}{\partial y_j}(a,\mathbf{v}_{ij})\right)$,
and from the properties of the $\bfv_{ij}$ we see that
for all $i,j$ 
$$
\frac{\partial G_i}{\partial y_j}(a,\mathbf{v}_{ij})\ <\ 
\frac{\partial G_i}{\partial y_j}(a,\mathbf{b})
$$
since  each ${\partial G_i}/{\partial y_j}$ depends on all the variables
$x,y_1,\ldots,y_m$.

{}From these remarks and the monotonicity of $\Lambda$ one has
$$ 
1\ \le\ 
\Lambda\left( \frac{\partial G_i}{\partial y_j}(a,\mathbf{v}_{ij})\right) \ <\ 
\Lambda( a,\mathbf{b}),
$$
showing that  $(a,\mathbf{b}) \neq (\rho,\pmb{\tau})$ implies $\Lambda(a,\mathbf{b})>1$.

\end{proof}

\begin{definition}
A characteristic point $(a,\mathbf{b})$
is an
\em{eigenpoint} if
$\Lambda\big(a,\mathbf{b}\big)=1$.
\end{definition}

The following theorem summarizes the key results for well-conditioned systems.

\begin{theorem}\label{RCP Thm}
Let $\Sigma: \mathbf{y} = \bG(x,\mathbf{y})$ be a well-conditioned system.
Then the following hold:
\begin{thlist}
\item
$(\rho,\pmb{\tau})\in\domp(\bG)$
\item
If $(\rho,\pmb{\tau})$ is in the interior of $\domp(\bG)$ then it is an 
eigenpoint.
\item
The system $\Sigma$ has at most  one eigenpoint.
\item
If there is an eigenpoint of $\Sigma$ then it must be $(\rho,\pmb{\tau})$.
\item
If there is no eigenpoint of $\Sigma$  then $(\rho,\pmb{\tau})$ lies on the
boundary of $\domp(\bG)$ and one has
$\Lambda(\rho,\pmb{\tau})<1$.
\end{thlist}

\end{theorem}

This result can be superior to Proposition \ref{location} for 
computing purposes since the
latter requires that one know all
 characteristic points of $\Sigma$ before being able to isolate the one
candidate for $(\rho,\pmb{\tau})$. Theorem \ref{RCP Thm} says that if
 one can find a characteristic point $(a,\mathbf{b})$ with $J_{\mathbf{G}}(a,\mathbf{b})$ 
 having largest positive eigenvalue 1, it is $(\rho,\pmb{\tau})$. 
As with the 1-equation case, if there are no eigenpoints of $\Sigma$, then
new methods are needed.

Flajolet and Sedgewick do not make use of the theory of characteristic
points in their work on multi-equation systems in \cite{FlSe2009} beyond citing the work of Drmota. Instead, they consider the polynomial case in the general setting of arbitrary non-degenerate $m$-equation systems
$\mathbf{P}(x,\bfy) = 0$ in Chap.~VII. 

Let $\cC$ be the set of solution points $(a,\mathbf{b})\in \bbC^{m+1}$ of such a system.
The non-degeneracy condition implies that each 
$\cC_i := \{(a,b_i) : (a,\mathbf{b})\in\cC\}$ is an algebraic curve.  For such curves there is a simple procedure to find a finite set $X_i$
 of points $(a,b_i)$ such that all singularities of $\cC_i$ are in $X_i$.

When applying the general method of \cite{FlSe2009} to the special case of 
well-conditioned systems $\bfy = \bG(x,\bfy)$, to find 
the extreme point $(\rho,\pmb{\tau})$, one can bypass the considerable work 
 of (1) determining the branch points $(a,b_i)$ of the algebraic curves 
$\cC_i$ among the points in $X_i$, and then (2) studying the Puiseux expansions of 
branches of $\cC_i$ about these branch points. Instead one only needs to test the finitely many 
points in $\{(a,\mathbf{b}) : (a,b_i)\in X_i\}$ to see which is the eigenpoint of the system
--- this will be $(\rho,\pmb{\tau})$.

\section{Drmota's Theorem Revisited}\label{Drmota sec}

In 1993 Lalley \cite{Lalley1993} proved that the solutions $y_i =T_i(x)$ to a well-conditioned
\textit{polynomial} system $ \mathbf{y} = \bG(x,\mathbf{y})$ would have a square-root singularity 
at $\rho$, and thus one had the familiar P\'olya asymptotics for
the coefficients.\footnote{Having a polynomial system is a very strong condition since it immediately
tells you that $\rho$ is a branch point, which leads to a Puiseux expansion;  it is only a matter of determining the order of the branch point (which is nonetheless a nontrivial task).}
In 1997 \cite{Drmota1997}, and again in 2009 \cite{Drmota2009}, Drmota presented the first 
sweepingly general theorem concerning
the asymptotic behavior of the coefficients of solutions of a well-conditioned system, namely 
the coefficients will again satisfy the same law  that P\'olya found
 to be true for several classes of trees (see \cite{PoRe1987}).  However, as explained in 
 Footnote \ref{Drmota footnote}, the hypotheses that Drmota has for the characteristic points
of the system seem to be incorrect in the first publication, and vague in the second.\footnote
{
The book \cite{FlSe2009} gives a detailed study of well-conditioned polynomial systems, but
only states the result for general well-conditioned systems. This statement is the 1997 version of Drmota's theorem, including the error in the hypotheses. The simplest patch is to replace the condition that `some characteristic point $(a,\mathbf{b})$ is in the interior of the domain' with 
the requirement that `$(\rho,\pmb{\tau})$ is in the interior of the domain'.
} 
To prove the theorem one needs to be able to show that $(\rho,\pmb{\tau})$ is in the 
interior of the domain of $\bG(x,\mathbf{y})$.
The following subsection gives a clear statement of the hypotheses needed, along with a slightly different 
proof of the key induction step for the proof.

\subsection{Drmota's Theorem}\label{Drmota revisited}
The following version is somewhat simpler than that presented by Drmota since there 
are no parameters.

\begin{theorem}\label{drmota}
Let $\Sigma: \mathbf{y} = \bG(x,\mathbf{y})$ be a well-conditioned system with standard
solution $\mathbf{T}(x)$.
Suppose  $ \Sigma$ has an eigenpoint $(\rho, \pmb{\tau})$ in the interior of $\domp(\bG)$. Then each $T_i(x)$ is the standard solution to a 
well-conditioned 1-equation system $y_i = \widehat{G}_i(x,y_i)$ with $(\rho,\tau_i)$ in the interior of $\domp(\widehat{G}_i)$. 
Thus each $T_i(x)$
has a square-root singularity at $\rho$, and the familiar P\'olya asymptotics (see, e.g., \cite{BBY2006}) hold for
the non-zero coefficients. 
\end{theorem}

\begin{proof}
One only needs to consider the case that the system has at least two equations,
and one can assume all second partials of the $G_i$ with respect to the $y_j$ are non-zero.
The following shows that 
eliminating the first equation (and $y_{1}$) yields a well-conditioned 
system with one less equation which has the standard solution
$\big(T_2(x),\ldots,T_m(x)\big)$ and an eigenpoint in the interior of the 
domain of the system.

By the Implicit Function Theorem one can solve the first equation
$$
y_{1}\ =\ \bG_{1}(x,\mathbf{y})
$$
for $y_{1}$, say
$$
y_{1}\ =\  H_{1}(x,y_{2},\ldots,y_{m}),
$$
where $H_{1}$ is holomorphic in a neighborhood of the origin, that is,
$H_{1}(0,\mathbf{0})=0$ and 
$$
H_{1}(x,y_{2},\ldots,y_{m}) \ =\ 
 G_{1}\big(x,H_{1}(x,y_{2},\ldots,y_{m}),y_{2},\ldots,y_{m}\big)
 $$
 in a neighborhood of the origin.
 
 Since the $T_{i}(x)$ take small values near the origin (as they are continuous functions that vanish at $x=0$), it follows that
 $$
H_{1}\big(x,T_{2}(x),\ldots,T_{m}(x)\big) \ =\ 
 G_{1}\Big(x, H_{1}\big(x,T_{2}(x),\ldots,T_{m}(x)\big),T_{2}(x),\ldots,T_{m}(x)\Big)
 $$
 holds in a neighborhood of the origin.
 Also one has 
 $$
 T_{1}(x)\ =\  G_{1}\big(x,T_{1}(x),T_{2}(x),\ldots,T_{m}(x)\big)
 $$ 
 holding in a neighborhood of the origin, so by the uniqueness of solutions
 in such a neighborhood, we must have
 $$
  T_{1}(x)\ =\  H_{1}\big(x,T_{2}(x),\ldots,T_{m}(x)\big)
 $$ 
in a neighborhood of the origin. By Proposition \ref{permanence}, this equation
actually holds globally for $|x| \le \rho$; in particular $H_{1}$ converges at
$(\rho,\tau_{2},\ldots,\tau_{m})$. 
By Corollary \ref{prop of Lambda}(a)
 the Jacobian $1-\dfrac{\partial G_1}{\partial y_1}$
 of the equation $y_{1}\ =\ G_{1}(x,\mathbf{y})$ does not vanish at $(\rho,\pmb{\tau})$.
Thus, by the Implicit Function Theorem, $ H_{1}$ is holomorphic at $\big(\rho,\tau_2,\ldots,\tau_m\big)$.
 
 Now discarding the first equation and substituting $H_1(x,y_2,\ldots,y_m)$ for $y_1$ in 
 the remaining equations gives a well-conditioned system of $m-1$ equations 
 \[y_i = G_i^\star(x,y_2,\ldots,y_m),\] $2\le i\le m$, with standard solution
 $\big(T_{2}(x),\ldots,T_{m}(x)\big)$ whose extreme point \[\big(\rho,\tau_2,\ldots,\tau_m\big)\] 
 is an eigenpoint, since it is a characteristic point of the system that is in the interior of 
 $\domp(\bG^\star)$. Thus the elimination procedure can continue if $\bG^\star$ consists of 
 more than one equation.

\end{proof}

The extreme point of a well-conditioned polynomial system, such as Example \ref{4 CP}, 
is always a characteristic point, and, as Lalley \cite{Lalley1993} proved, the coefficients 
of the solutions $T_i(x)$ have the classical P\'olya form
$C_i\rho^{-n}n^{-3/2}$. Drmota \cite{Drmota1997} extended Lalley's result to 
well-conditioned power series systems
with the extreme point in the interior of the domain of the system. A natural (and desirable) 
direction to consider for further research would be to drop the irreducible requirement. 
However, even in the polynomial case, this leads to substantial challenges, see Example \ref{ex reducible}. 

\subsection{A Wealth of Examples}
In \cite{BBY2006} we showed that single equation systems formed from a wide array of 
standard operators like Multiset, Cycle and Sequence led to square-root singularities and 
P\'olya asymptotics for the coefficients. The arguments used there easily carry over to the 
setting of systems of equations since the conditions in that paper force the positive domain 
to be an open set, and this guarantees that $(\rho,\pmb{\tau})$ is an interior point of the 
domain of the system, leading to a wealth of examples. 


\section{Some Open Problems about Characteristic Points of Well-Conditioned Systems}

\begin{question}
How can one locate $(\rho,\pmb{\tau})$ if it is not a characteristic point?
\end{question}

\begin{question}\label{CP finite Q}
Is the set of characteristic points always finite?
\end{question}
As one can see in the examples, Appendix \ref{ex section}, a system
 can have multiple characteristic points; the two equation polynomial system in
 Example \ref{4 CP} has four characteristic points. Example \ref{real solutions}
shows that the set of \textit{real} solutions to the characteristic system need not  be finite.
However Question \ref{CP finite Q} asks if the set of \textit{positive} solutions is finite.


\appendix
\section{A Collection of Basic Examples}\label{ex section}

The following examples explore the behavior of characteristic points of well-conditioned systems---the computational steps have been omitted. However the reader can find complete details online in the original preprint
\cite{BBY Arxiv}. 

    
 \subsection{Examples for 1-equation systems}
 
For 1-equation systems the following two examples show the three kinds of possible behavior, namely: 
(i) there is a characteristic point which is an interior point and thus
equal to $(\rho,\tau)$, 
(ii) there is a characteristic point which is a boundary point and  thus equal to 
$(\rho,\tau)$, and 
(iii) there is no characteristic point. 
If $(\rho,\tau)$ is in the interior  of the domain of $G$ then 
$x=\rho$ is a square-root singularity of $T(x)$.\footnote
{The possibilities for the nature of this singularity when $(\rho,\tau)$ is on the boundary 
of the domain of $G$ have not been classified.
Examples constructed along the lines of Proposition \ref{comp prop} show that one
can have $2^k$-root singularities. Comments VI.18 and VI.19 on p.~407 of \cite{FlSe2009}
state that one can have $\alpha$-root singularities, for $1 < \alpha \le 2$.
}

Each example starts with an equation $y=G(x,y)$ where the characteristic point 
$(\rho,{\tau})$ is in the interior of the domain of $G(x,y)$. 
Then the example is modified to give a system
$y=G^\star(x,y)$ with $(\rho^\star,\tau^\star)$ on the boundary of the domain 
of $G^\star(x,y)$. $(\rho^\star,\tau^\star)$ is  a characteristic point in 
Example \ref{first 1-example} 
but not in Example \ref{second 1-example}.


\begin{example}\label{first 1-example}
Let $G(x,y) = x(1+y^2)$.
For the characteristic system
$$
\bigg\{
\begin{array}{r c l}
y&=&x(1+y^2)\\
1&=&2xy
\end{array}
$$
of $y=G(x,y)$ one has  the characteristic point $(1/2,1)$, an interior point of 
the domain of $G(x,y)$, so for the standard solution
$y=S(x)$ of $y=G(x,y)$ 
 one has $(\rho,\tau) = (1/2,1)$ . The established theory for such 
a system (see \cite{FlSe2009}, Chapter VII) shows that $S(x)$ has a square-root singularity at $x=\rho$.

Next let $G^\star(x,y) =S(x)(1 + y^2)/2$.
For the characteristic system
$$
\bigg\{
\begin{array}{r c l}
y&=&S(x)(1+y^2)/2\\
1&=&S(x)y
\end{array}
$$
once again the characteristic point is $(1/2,1)$,
 but now it is a boundary point of the domain of $G^*(x,y)$.
An examination of the standard solution (see Proposition \ref{comp prop}) of $y=G^*(x,y)$, 
namely  $y = T(x) = S\big(S(x)/2\big)$, shows that it has 
a fourth-root singularity at $x=1/2$.
\end{example}

\begin{example} \label{second 1-example}
Let $G(x,y) = x\big(1 + 2 y + 2y^2 \big)$.
The characteristic system
$$
\bigg\{
\begin{array}{r c l}
y&=&x\big(1 + 2y + 2y^2\big)\\
1&=&2x(1+2y)
\end{array}
$$
of $y=G(x,y)$ has the characteristic point
$$\bigg(\frac{\sqrt{2} - 1}{2},\frac{\sqrt{2}}{2}\bigg),$$
 an interior point of the domain of $G(x,y)$, so for the standard solution
$y=S(x)$ of $y=G(x,y)$  one has $\rho = \big(\sqrt{2} - 1\big)/{2}$ 
and $\tau = {\sqrt{2}}/{2}$. 
 $S(x)$ has a square-root singularity at $x=\rho$.

Next let $G^\star(x,y) =x\big(1 + S(x) + y + 2y^2\big)$. 
The standard solution of $y=G^\star(x,y)$ is 
again $y=S(x)$, so $(\rho^*,\tau^*)=(\rho,\tau)$.
The characteristic system
$$
\bigg\{
\begin{array}{r c l}
y&=&x\big(1 + S(x) + y + 2y^2\big)\\
1&=&x(1+4y)
\end{array}
$$
of $y=G^\star(x,y)$ has no characteristic point since the only candidate is $(\rho,\tau)$ and 
$$\rho(1+4\tau)=(1/2)\big(\sqrt{2}-1\big)\big(1+2\sqrt{2}\big)\neq 1.$$
$(\rho,\tau)$ is a boundary point
of the domain of $G^*(x,y)$ whose location is not detected by the method of characteristic points.
\end{example}

\begin{remark} On p.~83 of their  1989 paper \cite{MeMo1989} Meir and Moon offer an interesting example of a 1-equation system
without a characteristic point, namely $y = A(x)e^y$ where $A(x)=(1/6)\sum_n x^n/n^2$.
The characteristic system is
$$
y\ =\ A(x)e^y,\quad 1\ =\ A(x)e^y,
$$
so a characteristic point $(a,b)$ must have $b=1$, $A(a)=1/e$. But $1/e$ is not in the
range of $A(x)$, so there is no characteristic point. One can nonetheless easily find $(\rho,\tau)$ in this case since $(\rho,\tau)$ must lie on the boundary of the domain of $A(x)e^y$. Thus $\rho=1$, and then $\tau = A(1)e^\tau = (\pi^2/36)e^\tau$, so 
$\tau \approx 0.41529$.

The paper goes on to claim that by differential equation methods one can show that
the standard solution $y=S(x)$ has coefficient asymptotics $s(n)\sim C/n$.
However this cannot be true since such a solution would diverge at its radius of convergence $\rho=1$
(see \cite{BBY2006}), whereas the given equation $y=A(x)e^y$ is nonlinear in $y$, so the solution
must converge at $\rho$.
\end{remark}

\subsection{1-equation framework}\label{1frame}

This subsection gives a framework for 1-equation examples which will be
useful for building the 2-equation examples in $\S$\ref{2family}.

\begin{proposition}
Let $A(x)$ be the standard solution of
\begin{equation}\label{A eqn}
y\ =\ x(1 + \fa y + \fb y^2)
\end{equation}
where $\fa \ge 0$ and $\fb >0$. Then the following hold:
\begin{thlist}
\item
\[
A(x) \ =\ \frac{1}{2\fb x}\Big((1-\fa x)  - \sqrt{(1-\fa x)^2 - 4\fb x^2}\Big).
\]
\item
$A(x)$ has non-negative coefficients.
\item
A sufficient condition for $A(x)$ to have integer coefficients is that $\fa$ and $\fb$ are integers. 
\item
$A(x)$ has a positive radius of convergence $\rho_A$ given by
\[
\rho_A \ =\ \frac{1}{\fa + 2\sqrt{\fb}}.
\]
\item
$\tau_A := A(\rho_A)$ is finite and is given by
$$
\tau_A\ =\ \frac{1}{\sqrt{\fb}}.
$$
\item
$\rho_A$ is a square-root branch point of the algebraic curve defined by \eqref{A eqn}.
\item
$(\rho_A,\tau_A)$ is the unique characteristic point of \eqref{A eqn}, that is, it is the unique
positive solution $(x,y)$ to
\begin{eqnarray*} 
y&=&x(1 + \fa y + \fb y^2)\\
1&=&x(\fa + 2\fb y). 
\end{eqnarray*}
\end{thlist}
\end{proposition}
\begin{proof}
(Exercise.)
\end{proof}

\begin{proposition} \label{comp prop}
Given $\fa,\fc\ge 0$ and $\fb,\fd>0$ let $A(x)$ be the standard solution of
$$
y\ =\ x(1 + \fa y + \fb y^2)
$$
and  let $S(x)$ be the standard solution of
\[
y\ =\ x(1 + \fc y + \fd y^2).
\]
Let $T(x)$ be the standard solution of
\[
y\ =\ A(x)(1 + \fc y + \fd y^2).
\]
Then the following hold:
\begin{thlist}
\item
$
T(x) \ =\ S(A(x)).
$
\item
$
T(x) \ =\ \dfrac{1}{2\fd A(x)}\Big((1-\fc A(x))  - \sqrt{(1-\fc A(x))^2 - 4\fd A(x)^2}\Big).
$
\item
$T(x)$ has non-negative coefficients.
\item
A sufficient condition for $T(x)$ to have integer coefficients is that $\fa,\fb,\fc,\fd$ are integers.
\item 
If
$\sqrt{\fb} \ =\ \fc + 2\sqrt{\fd}$ then 
$$(\rho_T,\tau_T)\ =\ (\rho_A,\tau_S)\ =\  \Big(\frac{1}{\fa + 2\sqrt{\fb}},\frac{1}{\sqrt{\fd}}\Big),$$
and $T(x)$ has a fourth-root singularity at $\rho_T$.

\end{thlist}
\end{proposition}
 \begin{proof}
(Exercise.)
\end{proof}
 
 The restriction $\sqrt{\fb} \ =\ \fc + 2\sqrt{\fd}$
 is called the \textbf{critical composition condition} \mbox{\rm (CCC)}; this is
 the condition needed for $T(x) = S(A(x))$ to be a critical composition (as defined by 
 Flajolet and Sedgewick \cite{FlSe2009}, p.~411).

\subsection{Multi-equation systems} \label{2family}


\begin{proposition} \label{main example thm}
Suppose
$$
\fa,\fc_1 \ge 0,\quad\fb, \fc_2,\fd>0,\quad \sqrt{\fb} = \fc  +2\sqrt{\fd},\quad \fc = \fc_1 + \fc_2.
$$
Let $A(x)$, $S(x)$, and $T(x)$ be as in Proposition \ref{comp prop}
Then the following hold:
\begin{thlist}
\item
The quadratic system
$$
(SYS):\qquad
\left\{
\begin{array} {c c l}
y_1 &=& A(x)\big (1 + \fc_1 T(x) + \fc_2 y_2 + \fd {y_1}^2\big)\\
y_2 &=& A(x) \big(1 + \fc_1 T(x) +  \fc_2 y_1 + \fd {y_2}^2\big)
\end{array}
\right.
$$
is well-conditioned, and the standard solution
is $y_1 = y_2 = T(x)$.
\item
The extreme point $(\rho,\tau,\tau)$ of $(SYS)$ is given by
$$(\rho,\tau,\tau)\ =\ \bigg(\frac{1}{\fa + 2\sqrt{\fb}},\frac{1}{\sqrt{\fd}},\frac{1}{\sqrt{\fd}}\bigg).$$
It is on the boundary of the domain of $(SYS)$.
\item
$T(x)=S(A(x))$ has a fourth-root singularity at $x=\rho$.
 \item
 A positive point $(x,y,y)$ is a characteristic point of $(SYS)$ 
  iff either
$$
(\star)\qquad\left\{
\begin{array}{ c c l}
1 &=&  A(x)\bigg(\fc_2 + 2\sqrt{\fd\big(1 + \fc_1 T(x)\big)}\;\bigg)\\
y &=& \dfrac{1 - \fc_2 A(x)}{2\fd A(x)}
\end{array} 
\right.
$$
or
$$
(\star\star)\qquad\left\{
\begin{array}{ c c l}
1 &=&  A(x)\bigg(\fc_2 +2 \sqrt{{\fc_2}^2  + \fd\big(1+\fc_1T(x)\big)}\;\bigg)\\
y &=& \dfrac{1 + \fc_2 A(x)}{2\fd A(x)}.
\end{array}
\right.
$$
 
 \item
If $\fc_1 = 0$ then there are exactly two characteristic points of the form $(x,y,y)$: 
the first is $(\rho,\tau,\tau)$, a boundary characteristic point obtained from $(\star)$, 
and the second is the
unique positive solution to $(\star\star)$, an interior characteristic
point.  
This is the only
case where $(\star)$ contributes a characteristic point, namely $(\rho,\tau,\tau)$, and
this is the only case where $(\rho,\tau,\tau)$ is a characteristic point.
\item\label{c1=2c2}
If $0 < \fc_1 =2\fc_2$ then there is a unique characteristic point of the form $(x,y,y)$: 
it is the unique positive solution to $(\star\star)$ and it is a boundary point different
{}from $(\rho,\tau,\tau)$.
\item\label{c1<2c2}
If $0<\fc_1<2\fc_2$ then  there is a unique characteristic point of the form $(x,y,y)$: 
it is the unique positive solution to $(\star\star)$ and it is an interior point that is
different from $(\rho,\tau,\tau)$.
\item
If $2\fc_2<\fc_1$ then  there are no characteristic points of the form $(x,y,y)$, so
again $(\rho,\tau,\tau)$ is not a characteristic point.
\item The second characteristic point in (e) and the unique characteristic points in (f) and (g) are given explicitly by
  \begin{align*}
    x & = \frac{\fc+\sqrt{\fc^2+\ff}}{\fa\fc + 2\fc^2 + \ff + \fb + (\fa+2\fc)\sqrt{\fc^2+\ff}}\\
    y & = \frac{\fc+ \fc_2 + \sqrt{\fc^2+\ff}}{2\fd}
  \end{align*} where
\[
  \ff = -6\fc_1\fc_2 + 3\fc_2^2 + 4\fd.
\]
\end{thlist}

\end{proposition}

\begin{proof}
(Exercise.)

\end{proof}



Now we look at three well-conditioned examples that show 
some of the 
varied behavior of characteristic points 
when one has more than one equation in the system. 
In the first example there are two characteristic points, 
both in the interior of the domain of $\bG(x,\mathbf{y})$ 
and one of them is $(\rho,\pmb{\tau})$. In the second example
one has a characteristic point in the interior of the domain of $\bG(x,\mathbf{y})$ and
$(\rho,\pmb{\tau})$ is a characteristic point on the boundary of the domain. 
In the third example one has a characteristic point in the interior of the domain 
of $\bG(x,\mathbf{y})$ but $(\rho,\pmb{\tau})$ is not a characteristic 
point. In the second and third examples, $\rho$ is not a square-root singularity
of the solutions.
Such examples show the need 
for a more subtle use of characteristic points in the pursuit of information on
$(\rho,\pmb{\tau})$ for multi-equation systems.

\begin{example}\label{first example}
For the system of two equations
\begin{eqnarray*}
y_1&=&x\cdot\big(1+y_2+2y_1^2\big)\\
y_2&=&x\cdot\big(1+y_1+2y_2^2\big)
\end{eqnarray*}
add
$$(1-4xy_1)(1-4xy_2) - x^2 \ =\ 0$$
to obtain the characteristic system. This is a polynomial system, so all characteristic
points will be in the interior of the domain; and since $(\rho,\tau_1,\tau_2)$ is also in the interior
it must be a characteristic point. Let $(a,b,c)$ be a characteristic point. By a computation we see that $b\neq c$ is impossible. 
Thus the characteristic points are the positive triples $(a,b,b)$ satisfying
\begin{eqnarray*}
b&=&a\big(1+b+2b^2\big)\\
a^2&=&(1-4ab)^2.
\end{eqnarray*}
{}From this the system has two characteristic points:
\begin{eqnarray*}
\bigg(\frac{2\sqrt{2}-1}{7},\frac{1}{\sqrt{2}},\frac{1}{\sqrt{2}}\bigg)& \approx&
(0.2612,0.7071,0.7071)\\
\bigg(\frac{2 \sqrt{3}-1}{11}, \frac{1+\sqrt{3}}{2},\frac{1+\sqrt{3}}{2}\bigg) &\approx&
(0.2240,1.3660,1.3660).
\end{eqnarray*}
Now we are left with determining which of the two characteristic points is 
$(\rho,\tau_1,\tau_2)$. By applying either Proposition \ref{location} or 
Proposition \ref{Lambda lemma},  it is the first of these.
\end{example}

\begin{example}\label{counter Drmota C1=0}
Let $\fa=0$, $\fb=9$, $\fc_1 = 0$, $\fc_2=1$, and $\fd = 1$. 
These numbers satisfy  \mbox{\rm(CCC)}. Following the hypotheses of
Proposition \ref{main example thm},
let $A(x)$ be the standard solution to $y = x(1 +  9y^2)$ and consider
the system
\begin{eqnarray*}
y_1&=&A(x)\cdot\big(1+y_2 + y_1^2\big)\\
y_2&=&A(x)\cdot\big(1+y_1+y_2^2\big).
\end{eqnarray*}
Since $c_1=0$ there are two characteristic points of the form $(a,b,b)$. The first
is the extreme point
$$
(\rho,\tau_1,\tau_2)\ =\ (1/6,1,1)
$$
which lies on the boundary of the domain, and the second is the interior point
obtained from the formulas in Proposition \ref{main example thm} (i):
$$
\bigg(\frac{1 + 16\sqrt{2}}{146}, 1 + \sqrt{2}, 1 + \sqrt{2}\bigg).
$$

\end{example}

\begin{example}\label{counter Drmota C1>0}
Let $\fa=0$, $\fb=16$, $\fc_1 = 1$, $\fc_2=1$, and $\fd = 1$. 
These numbers satisfy  \mbox{\rm(CCC)}. Following the hypotheses of
Proposition \ref{main example thm},
let $A(x)$ be the standard solution to $y = x(1 +  16y^2)$, and
let $T(x)$ be the standard solution to $y = A(x)(1 +2y +y^2)$.
Consider
the system
\begin{eqnarray*}
y_1&=&A(x)\cdot\big(1+T(x) + y_2 + y_1^2\big)\\
y_2&=&A(x)\cdot\big(1+T(x) + y_1+y_2^2\big).
\end{eqnarray*}
Since $0<c_1<2c_2$, the extreme point 
$$
(\rho,\tau_1,\tau_2)\ =\ (1/8,1,1)
$$
is not a characteristic point, but there is a characteristic point
 of the form $(a,b,b)$ in the interior of the domain of $\bG$
 given by the formulas of Proposition \ref{main example thm}:
 $$
 (a,b,b)\ = \ \bigg(\frac{30 + 17\sqrt{5}}{545}, \frac{3+\sqrt{5}}{2}, \frac{3+\sqrt{5}}{2} \bigg).
$$
 
\end{example}

\subsection{Other examples}
The next example shows some characteristic points which are not of the form $(x,y,y)$
\begin{example}\label{4 CP}
The well-conditioned polynomial system
\begin{eqnarray*}
y_1&=& G_1(x,y_1,y_2)\  := \ x(1+2y_1^3+2x^3y_1^3y_2)\\
y_2&=& G_2(x,y_1,y_2)\  :=\  x(1+x^3y_2+2y_1^3y_2^2)
\end{eqnarray*}
has four characteristic points which, to 6 places of accuracy are:
\begin{quote}\sf
(0.1818598, 1.556545, 0.3647603)\\
(0.2640956, 1.210710, 0.5353688)\\
(0.3867644, 0.6661246, 3.834789)\\
(0.4153198, 0.6217456, 0.4743552)
\end{quote}
One sees that these four points form an antichain, as required by Proposition \ref{antichain}.
The extreme point $(\rho,\tau_1,\tau_2)$ of a polynomial system is a characteristic point.
By Proposition \ref{location} it must be the last one since it has the largest $x$-value,
assuming one has found all characteristic roots of this system.
If one is not sure that there are only four characteristic points then, 
by Theorem \ref{RCP Thm}, it suffices to verify that the indicated characteristic point is an eigenpoint.
\end{example}

This example demonstrates that iteration is not sufficient to obtain a new systems 
$\Sigma^\star$ such that the
Jacobian matrix $J_{\bG^\star}(x,\bfy)$ has non-zero entries.
\begin{example}\label{ex subst}
Consider the irreducible system  $ \mathbf{y} = \bG(x,\mathbf{y})$  of 4 equations:
$$
\Sigma \ = \ \left\{
\begin{array}{ r c l}
y_1& =&G_1(x,y_1,\ldots,y_4)\ :=\  x\big(1 + {y_2}^2 + {y_4}^2\big)\\
y_2& =&G_2(x,y_1,\ldots,y_4)\ :=\  x\big(1 + {y_1}^2 + {y_3}^2\big)\\
y_3&=&G_3(x,y_1,\ldots,y_4)\ :=\  x\big(1 + {y_4}^2\big)\\
y_4&=&G_4(x,y_1,\ldots,y_4)\ :=\  x\big(1 + {y_1}^2\big).
\end{array}\right.
$$
Let $M=J_{\bG^{(n)}}$. 
Then it is easy to check that $M_{11}\neq 0$  
 iff  $n$ is odd, and $M_{12}\neq 0$  iff $n$ is even. 
 Thus for $n\ge 1$, $J_{\bG^{(n)}}(x,\mathbf{y})$ has entries
 which are 0.

One can transform $\Sigma$ into a system $\Sigma^\star$ where the Jacobian of 
$\bG^\star$ has all entries non-zero by doing {\em selective} substitutions. 
For example, in the first equation of $\Sigma$ replace {\em one} of the two $y_2$'s 
by $ G_2(x,\mathbf{y})$, giving the system
$$
 \left\{
\begin{array}{ r c l}
y_1& =& x(1 + y_2 G_2(x,\mathbf{y} ) + {y_4}^2)\\
y_2& =& x(1 + {y_1}^2 + {y_3}^2)\\
y_3&=& x(1 + {y_4}^2)\\
y_4&=& x(1 + {y_1}^2)
\end{array}\right.
$$
The first equation in this system is such that the right hand side 
depends on all 4 of the $y_i$. Continuing in this manner one obtains a system 
in which every $ G_i(x,\mathbf{y})$ depends on each of $y_1,\ldots,y_4$.
\end{example}

This example shows complications which can arise with reducible systems.
\begin{example}\label{ex reducible}
Consider the reducible polynomial system
\begin{eqnarray*}
y_1&=&y_3\cdot\big(1+y_2 + y_1^2\big)\\
y_2&=&y_3\cdot\big(1+y_1+y_2^2\big)\\
y_3&=&x\cdot (1 +  9{y_3}^2).
\end{eqnarray*}
Let the third equation have the standard solution $y_3=A(x)$. One then sees that this example is really 
just an alternate presentation of Example \ref{counter Drmota C1=0} where the solutions for $y_1$ and 
$y_2$ have a fourth-root singularity at their radius of convergence. 
\end{example}

This final example shows that there can be infinitely many real solutions to a characteristic system, in contrast to what has been observed so far for characteristic points, see Question \ref{CP finite Q}. 
\begin{example}\label{real solutions}
For the characteristic system (belonging to a 2-equation system)
$$ 
\left\{
\begin{array}{l c l}
       y_1-x\cdot\big(1+ y_1+ y_1y_2 \big) &=& 0\\
       y_2-x\cdot\big(1+ y_2+ y_1y_2 \big) &=&0\\
       (x-1)\cdot \big(x+xy_1+xy_2 - 1\big)   &=&0
\end{array} \right.
$$
the real solutions include the infinite curve
$$
\big\{(x,y_1,y_2) : x=1, y_1y_2 = -1\big\}.
$$
\end{example}

\section{Background Material}\label{background}


\subsection{The extended nonnegative real numbers}
\quad\\
Extend the usual operations on $[0,\infty)$ to $[0,\infty]$ in the obvious way as follows:
\begin{eqnarray*}
c+\infty&=&\infty\quad \text{for }c\in[0,\infty]\\
c\cdot \infty&=&\infty\quad \text{for }c\in(0,\infty]\\
\sum_n c_n&=&
\begin{cases}
\text{the usual infinite sum}&\text{ if all }c_n\in [0,\infty)\\
\infty&\text{if some }c_n = \infty.
\end{cases}
\end{eqnarray*}
Here the \textit{usual} infinite sum is $\infty$ if the series diverges.  Note that $0\cdot \infty$ is left undefined since it is indeterminate.


\subsection{Formal power series in several variables}

This section gives the essential definitions that lay the foundations for 
working with formal power series in several variables. 
The standard 
number systems are:
\begin{quote}
the set $\mathbb{N}=\{0,1,\ldots\}$ of {\em nonnegative} integers,
the set $\mathbb{Q} $ of {\em rational} numbers,
the set $\mathbb{R}$ of {\em real numbers}, and
the set $\mathbb{C}$ of {\em complex} numbers.
\end{quote}

For the linearly ordered set $\mathbb{R}$ of real numbers one has
the posets of \textit{real-valued functions on} $X$, where the
partial ordering is given by $f\leq g$ if $f(x)\le g(x)$ for all $x\in X$.
Familiar examples are:
\begin{thlist}
\item
$n$-\textit{vectors} $\mathbf{v} = (v_1,\ldots,v_n)$, by setting $X =\{1,\ldots,n\}$
\item
$m\times n$-\textit{matrices} $M$, by setting $X =\{1,\ldots,m\}\times \{1,\ldots,n\}$
\item
\textit{formal power series in $k$-variables} $A(x_1,\ldots,x_k)$ by setting
$X = \mathbb{N}^{k}$. In this case a function $a$ from $\mathbb{N}^k$ to $\mathbb{R}$ provides the 
coefficients, and one writes 
$$
A(\mathbf{x})\ :=\ \sum_{\mathbf{i}\in \mathbb{N}^k} 
a(\mathbf{i}) \mathbf{x}^\mathbf{i}
$$
\end{thlist}
A matrix (or vector) $M$ of real numbers is \textit{non-negative} (written $M\ge 0$) 
if each entry is non-negative, and \textit{positive} (written $M>0$) if each entry is positive. 
A power series $A(\mathbf{x})$ is \textit{non-negative} (written $A(\mathbf{x})\ge 0$) if each coefficient is non-negative.
 
 
 \subsubsection{Composition of formal power series}
 
 For power series $A(w_{1},\ldots,w_{m})$ and $B_{\ell}(\mathbf{x})$, $1\le \ell \le m$, 
 where the constant term of each $B_{\ell}$ is zero, that is,  $b_{\ell}(\mathbf{0})=0$, define the formal
{\em composition}
$$C(\mathbf{x})\ :=\ A\big(B_{1}(\mathbf{x}),\ldots,B_{m}(\mathbf{x})\big)$$ 
by defining the coefficient function as follows:
$$
c(\mathbf{i})\ :=\ 
\sum_{\mathbf{j}\ge\mathbf{0}}
\big[\mathbf{x}^{\;\mathbf{i}}\big] \;
a(\mathbf{j}) \cdot {B_{1}(\mathbf{x})}^{j_{1}}\cdots {B_{m}(\mathbf{x})}^{j_{m}}
$$
Requiring that the constant term of the $B_{\ell}(\mathbf{x})$ be 0 guarantees that for each $\mathbf{i}$ 
only finitely many terms in this sum are nonzero. Consequently $C(\mathbf{x})$ is indeed
a formal power series.


  \subsubsection{The function defined by a formal power series}
  
A power series $A(\mathbf{x})$ in $k$ variables defines a partial function,
also denoted $A(\mathbf{x})$,
on $\mathbb{R}^k$ (or $\mathbb{C}^k$) by setting
 \begin{equation}\label{def PS fcn}
  {A}(\mathbf{c})\ :=\ 
  \sum_{n\ge 0}\ \sum_{i_{1}+\cdots+i_{k}\ =\ n}a(\mathbf{i}){\mathbf{c}^\mathbf{i}}
  \qquad (\mathbf{c}\in \mathbb{R}^k)
  \end{equation}
  whenever the sum converges.

  For  $A(\mathbf{x})$  a nonnegative power series in $k$ variables, 
 and for $\mathbf{c}\in[0,\infty]^{k}$,
  $A(\mathbf{c})\ =\ \infty$
  if the series  \eqref{def PS fcn} diverges, that is, if
  $$
  \lim_{n\rightarrow\infty}\ \sum_{{j\le n}}\ \sum_{i_{1}+\cdots+i_{k}\ =\ j}
  a(\mathbf{i}) \mathbf{c}^\mathbf{i}\ =\ \infty.
  $$
  A nonnegative power series $A(\mathbf{x})$ in $k$ variables defines a left-continuous function  from 
$[0,\infty]^k$ to $[0,\infty]$ and is monotone nondecreasing in each variable on $[0,\infty]^k$.

 
 \subsubsection{The derivatives of a formal power series}
 Derivatives of [nonnegative] formal power series give [nonnegative] formal power series:
 $$
 \frac{\partial A(\mathbf{x})}{\partial x_j}\ :=\ 
\sum_{\mathbf{i}\ge\mathbf{0}} i_j a(\mathbf{i})x_1^{i_1}\cdots x_j^{i_j -1}\cdots x_k^{i_k}.
 $$
  The notation $A_{x_{j}}$  is also used for the partial derivative $\partial A/\partial x_{j}$.


\subsubsection{Holomorphic functions and a law of permanence}

A complex-valued function $f(\mathbf{x})$ of several complex variables is {\em holomorphic} at $\mathbf{c}$
if it is continuous and differentiable in a neighborhood of $\mathbf{c}$.
The notation $[\mathbf{a},\mathbf{b}]$ is short for 
\[
[a_{1},b_{1}]\times\cdots\times[a_{k},b_{k}].
\]

\begin{proposition}[A Law of Permanence for Functional Equations]\label{permanence}
Suppose \[A(\mathbf{x}),B(\mathbf{x},y)\ge 0.\] If there is an $\varepsilon>0$ such
that
$$
A(\mathbf{x})\ =\ B\big(\mathbf{x},A(\mathbf{x})\big) \ <\ \infty\quad
\text{for } \mathbf{x}\in [\mathbf{0},\pmb{\varepsilon}]
$$
then
$$
A(\mathbf{x})\ =\ B\big(\mathbf{x},A(\mathbf{x})\big) \quad
\text{for } \mathbf{x}\in [\mathbf{0},\pmb{\infty}].
$$
If furthermore $\mathbf{a}\;>\mathbf{0}$   and 
$A(\mathbf{a}) \ <\ \infty$
then
$$
A(\mathbf{x})\ =\ B\big(\mathbf{x},A(\mathbf{x})\big) \quad
\text{for } |x_{i}| \le a_{i},\  1\le i\le k
$$
 and $A(\mathbf{x})$
is holomorphic for $ |x_{i}| < a_{i},\  1\le i\le k$.
\end{proposition}

\begin{proof}
This is a special case of Hille's \textit{law of permanence for functional equations}
given in $\S10.7$ of Vol.~2, \cite{Hille1962}.
\end{proof}

 
 \subsection{The Perron-Frobenius theory of nonnegative matrices}
The key to the main results of this paper are some simple observations based on the well-known Perron-Frobenius theory of nonnegative matrices that was developed 
ca.~1910.
\begin{proposition}\label{PF Prop}
Let $M$ be {an irreducible} nonnegative nonzero $k\times k$ matrix with real entries.
\begin{thlist}
  \item $M$ has a real eigenvalue.  
 \item The largest real eigenvalue $\Lambda(M)$ is positive and is given by
$$
\Lambda(M)\ =\ \max_{\mathbf{x}>\mathbf{0}} \min_{1\le i\le k}\frac{(M\mathbf{x})_i}{x_i}.
$$
  \item $\Lambda(M)$ is a\textit{ simple root} of the characteristic polynomial 
  $p_{M}(\lambda) = \det(\lambda I-M)$.
  \item The eigenspace belonging to $\Lambda(M)$ is 1-\textit{dimensional}, 
  generated by a unique positive normalized eigenvector $\mathbf{v}_{M}$. 
  ({Normalized} means the sum of the entries is 1).
\end{thlist}
\end{proposition}

\begin{proof}
(See  \S2 of Gantmacher \cite{Gant1959}.) 
\end{proof}

Note that Proposition \ref{PF Prop}(b) implies that for some $\mathbf{x}>\mathbf{0}$ 
one has $\Lambda(M)$ equal to $ \min_{1\le i\le k}\dfrac{(M\mathbf{x})_i}{x_i}$.


\begin{corollary} \label{prop of Lambda}
\begin{thlist}
  \item
  A positive $k\times k$ matrix $M$, $k\ge 2$, has all diagonal entries $<\Lambda(M)$.
  \item
  $\Lambda(X)$ is a nondecreasing function on the set of nonnegative matrices, that is,
  $M_1\leq M_2$ implies $\Lambda(M_1)\le \Lambda(M_2)$. Furthermore if every row [column] sum of $M_1$
  is less than the corresponding row [column] sum of $M_2$ then $\Lambda(M_1)< \Lambda(M_2)$.
  \item
  $\Lambda(X)$ is a continuous function on the set of nonnegative matrices, where the matrices
  are thought of as points in $k^2$-space.
  \end{thlist}
  \end{corollary}

  \begin{proof}(Exercise.)

(Note: A special case of item (c) is stated on p.~2103 of Lalley \cite{Lalley1993}, for certain Jacobian matrices denoted $J_z$, evaluated along certain curves.)
\end{proof}

\end{document}

%% file: symb1.tex
\newcommand{\domp}{{\sf Dom^{+}}}

\newcommand{\bbC}{{\mathbb C}}

\newcommand{\bbR}{{\mathbb R}}

\newcommand{\fa}{{\mathfrak a}}
\newcommand{\fb}{{\mathfrak b}}
\newcommand{\fc}{{\mathfrak c}}
\newcommand{\fd}{{\mathfrak d}}

\newcommand{\ff}{{\mathfrak f}}

\newtheorem{theorem}{Theorem}
\newtheorem{proposition}[theorem]{Proposition}
\newtheorem{lemma}[theorem]{Lemma}
\newtheorem{corollary}[theorem]{Corollary}

\newtheorem{definition}[theorem]{Definition}
\newtheorem{example}[theorem]{Example}

\newtheorem{question}{Question}

\newtheorem{remark}[theorem]{Remark}

 \newcounter{thlistctr}
 \newenvironment{thlist}{\
 \begin{list}%
 {\alph{thlistctr}}%
 {\setlength{\labelwidth}{2ex}%
 \setlength{\labelsep}{1ex}%
 \setlength{\leftmargin}{6ex}%
 \usecounter{thlistctr}}}%
 {\end{list}}

\newsymbol \ndiv   232D
\newsymbol \nmodels  2332

\newcommand{\bfv}{{\mathbf v}}

\newcommand{\bG}{{\mathbf G}}

\newcommand{\bT}{{\mathbf T}}

\newcommand{\cC}{{\mathcal C}}

\newcount\minute        
\newcount\hour          
\newcount\hourMins  
\def\now%
{
  \minute=\time    
  \hour=\time \divide \hour by 60 
  \hourMins=\hour \multiply\hourMins by 60
  \advance\minute by -\hourMins 
  \zeroPadTwo{\the\hour}:\zeroPadTwo{\the\minute}%
}
\def\timestamp%
{
  \today\ \now
}
\def\zeroPadTwo#1%
{
  \ifnum #1<10 0\fi    
  #1
}